\documentclass[11pt]{amsart}
\usepackage[top=1.5in, bottom=1.5in, left=1.4in, right=1.4in]{geometry} 
\usepackage{color}
\usepackage{graphicx}
\usepackage{amssymb}
\usepackage{epstopdf}
\usepackage{amsthm}
\usepackage{hyperref}
\usepackage[table]{xcolor}
\usepackage{array}
\usepackage{mathtools}
\usepackage{enumerate}
\usepackage{lscape}
\usepackage{color, colortbl}
\usepackage{tabularx}
\usepackage{longtable}
\usepackage{booktabs}
\usepackage{multicol}
\usepackage{relsize}

\allowdisplaybreaks

\usepackage[all]{xy}
\usepackage{graphicx}  
\usepackage{tikz}
\usetikzlibrary{decorations.pathreplacing,calc}
\usetikzlibrary{decorations.pathmorphing,backgrounds,positioning,fit,matrix}
\usepackage{tkz-graph}
\usetikzlibrary{arrows}
\usetikzlibrary{arrows.meta}
\usetikzlibrary{decorations.markings}

\numberwithin{equation}{section}

\usepackage{etoolbox}
\apptocmd{\sloppy}{\hbadness 10000\relax}{}{}

\newtheorem{thm}{Theorem}[section]
\newtheorem{lemma}[thm]{Lemma}

\newtheorem{cor}[thm]{Corollary}
\newtheorem{prop}[thm]{Proposition}
\newtheorem{conj}[thm]{Conjecture}

\newtheorem{problem}[thm]{Problem}

\newtheorem{Definition}[thm]{Definition}
\newenvironment{definition}
  {\begin{Definition}\rm}{\end{Definition}}

\newtheorem{Example}[thm]{Example}
\newenvironment{example}
  {\begin{Example}\rm}{\end{Example}}
  
\newtheorem{Remark}[thm]{Remark}
\newenvironment{remark}
  {\begin{Remark}\rm}{\end{Remark}}

\newtheorem{Question}[thm]{Question}
\newenvironment{question}
  {\begin{Question}\rm}{\end{Question}}

\def\<{\langle}
\def\>{\rangle}

\newcommand\raU[1]{{\xrightarrow[#1]{}}}
\newcommand\raAstU[1]{{\xrightarrow[#1]{\ast}}}

\def\ra{{\raU{}}}
\def\raAst{{\raAstU{}}}

\def\point{\mathsmaller{\Phi}}
\def\PP{\mathsmaller{\Phi^{+}}} 

\def\raPoint{{\xrightarrow[\point]{}}}
\def\raPP{{\xrightarrow[\PP]{}}}
\def\raPPAst{\xrightarrow[\PP]{\ast}}
\def\raPointAst{{\xrightarrow[\point]{\ast}}}

\def\pointprime{\mathsmaller{\Phi'^+}}
\def\raPPprime{{\xrightarrow[\pointprime]{}}}
\def\raPPprimeast{{\xrightarrow[\pointprime]{\ast}}}

\def\Span{{ \operatorname{Span}}_{\R}}

\newcommand\orbit[2]{{#1.#2}}
\newcommand\W[1]{{\orbit{W}{#1}}}
\newcommand\G[1]{{\orbit{G}{#1}}}

\def\FS{\operatorname{FS}} 
\def\SpanPP{{\FS_{\PP}}}
\def\SpanPoint{{\FS_{\point}}}

\def\DD{X} 
\def\DDaff{{\widetilde{\DD}}} 
\def\fund{\gamma} 
\def\extraRoot{0} 

\def\raFund{{\xrightarrow[\mathsmaller{\operatorname{UCF}}]{}}} 
\def\compFund{R} 
\def\compFundaff{{\widetilde{\compFund}}} 
\def\fl{\ell}

\def\sumcoordmap{\sigma} 
\newcommand\sumcoord[1]{\sigma(#1)} 
\newcommand\pseudostab[1]{{\tilde{#1}}} 
\def\pstv{{\pseudostab v}} 
\def\domin{\vartriangleleft} 
\def\sdomin{\blacktriangleleft} 
\def\allones{h} 

\def\R{\mathbb{R}}
\def\Z{\mathbb{Z}}

\def\init{\Omega\cup\{0\}}
\def\l{\lambda}
\def\m{\mu}

\def\scl{0.3}
\def\chipscl{0.6}
\def\chiptikzscl{0.7}
\def\chipcoef{1.2}

\newcommand{\chip}[1]{\begin{tikzpicture}[baseline={([yshift=4pt]current bounding box.south)},block/.style={draw,circle, minimum width={width("11")+12pt},
font=\normalsize,scale=\chipscl}]
\node[block] (A) at (0,0) {$#1$};
\end{tikzpicture}}

\def\sw{south}
\def\ne{north}
\def\scale{0.3}
\def\textScale{0.7}
\def\tpscale{1}
\def\sclbx{1}
\def\dheight{0.3}
\def\shadeofgrey{70}
\newcommand{\nodeC}[5]{
  \node[scale=\scale,draw=#5, fill=#5,circle,line width=1.5pt] (N#1) at (#3) { };
  \node[anchor=\ne,scale=\textScale,text=#5] (SS#1) at (N#1.\sw) {$\mathbf{#2}$};
  \node[anchor=\sw,scale=\textScale,text=#5] (NN#1) at (N#1.\ne) {$#4$};
}

\newcommand{\nodeN}[5]{
  \node[scale=\scale,draw=#5,circle,line width=0.25pt] (N#1) at (#3) { };
  \node[anchor=\ne,scale=\textScale,text=#5] (SS#1) at (N#1.\sw) {$#2$};
  \node[anchor=\sw,scale=\textScale,text=#5] (NN#1) at (N#1.\ne) {$#4$};
}

\newcommand{\nodeD}[5]{
  \node[scale=\textScale] (N#1) at (#3) {$\dots$};
  \node[anchor=\sw] (NN#1) at (N#1.\ne) {$#4$};
}

\def\centraltermination{Proposition~4.3}
\def\permcontainment{Proposition~2.2}
\def\permtrap{Lemma~8.2}
\def\permcc{Theorem~9.1}

\newcommand{\emailhref}[1]{\email{\href{#1}{#1}}}
  
\title{Root system chip-firing {I}{I}: central-firing}
\author{Pavel Galashin}
\emailhref{galashin@mit.edu}
\author{Sam Hopkins} 
\emailhref{samuelfhopkins@gmail.com}
\author{Thomas McConville}
\emailhref{thomasmcconvillea@gmail.com }
\author{Alexander Postnikov}
\emailhref{apost@math.mit.edu}
\address{Department of Mathematics, Massachusetts Institute of Technology, 77 Massachusetts Avenue,
Cambridge, MA 02139, USA}

\begin{document}

\begin{abstract}

Jim Propp recently proposed a labeled version of chip-firing on a line and conjectured that this process is confluent from some initial configurations. This was proved by Hopkins-McConville-Propp. We reinterpret Propp's labeled chip-firing moves in terms of root systems: a ``central-firing'' move consists of replacing a weight $\lambda$ by $\lambda+\alpha$ for any positive root $\alpha$ that is orthogonal to $\lambda$. We show that central-firing is always confluent from any initial weight after modding out by the Weyl group, giving a generalization of unlabeled chip-firing on a line to other types. For simply-laced root systems we describe this unlabeled chip-firing as a number game on the Dynkin diagram. We also offer a conjectural classification of when central-firing is confluent from the origin or a fundamental weight.

\end{abstract}

\date{\today}

\setcounter{tocdepth}{1}
\maketitle
\tableofcontents
\section{Introduction}

\emph{Chip-firing} is a certain (solitaire) game played on a graph that was introduced by Bj\"{o}rner, Lov\'{a}sz, and Shor~\cite{bjorner1991chip}. The states of this game are configurations of chips on the vertices of this graph. A vertex which has at least as many chips as neighbors is said to be \emph{unstable}. We can \emph{fire} any unstable vertex, which sends one chip from that vertex to each of its neighbors. And we can keep firing chips in this way until we reach a configuration where all vertices are stable. A fundamental result of Bj\"{o}rner-Lov\'{a}sz-Shor is that this process is \emph{confluent}: either we keep firing forever, or we reach a unique stable configuration that does not depend on which unstable vertices we chose to fire. As it turns out, this chip-firing process is essentially the same as the \emph{Abelian Sandpile Model}, originally introduced by the physicists Bak, Tang, and Wiesenfeld~\cite{bak1987self} and subsequently developed by Dhar~\cite{dhar1990self,dhar1999abelian}. For more on chip-firing and sandpiles, we refer the reader to~\cite{levine2010sandpile,corry2017divisors}.

 Bj\"{o}rner, Lov\'{a}sz, and Shor were motivated to define chip-firing on an arbitrary graph by earlier papers of Spencer~\cite{spencer1986balancing} and Anderson et al.~\cite{anderson1989disks} which studied the special case of chip-firing on a line, i.e., on an infinite path graph, which we denote by~$\Z$. Inspired by this initial setting, Jim Propp recently introduced a version of \emph{labeled} chip-firing on a line. The states of the labeled chip-firing process are configurations of distinguishable chips with integer labels $1,2,\ldots,N$ on~$\mathbb{Z}$. For example, with $N=4$, the following is such a configuration:
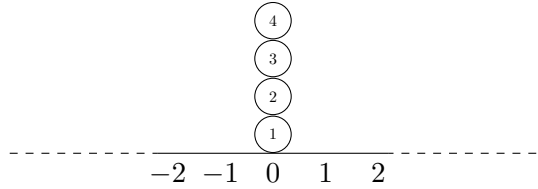
\begin{center}
\begin{tikzpicture}[scale=\chiptikzscl,block/.style={draw,circle, minimum width={width("11")+12pt},
font=\small,scale=\chipscl}]
    \foreach \x in {-2,-1,...,2} {%
      \node[anchor=north] (A\x) at (\x,0) {$\x$};
    }
    \draw (-2.2,0) -- (2.2,0);
    \draw[dashed] (-5,0) -- (-2.2,0);
     \draw[dashed] (5,0) -- (2.2,0);
    \foreach[count=\i] \a/\b in {0/1,0/2,0/3,0/4} {%
      \node[block] at (\a,{\b*\chipscl*\chipcoef-0.5*\chipscl*\chipcoef}) {$\i$};
      }
\end{tikzpicture}
\end{center}
The firing moves consist of choosing two chips that occupy the same vertex and moving the chip with the lesser label one vertex to the right and the chip with the greater label one vertex to the left. For example, if we chose to fire chips~\chip1 and~\chip3 in the previous configuration that would lead to:
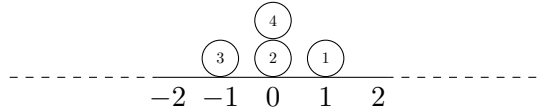
\begin{center}
\begin{tikzpicture}[scale=\chiptikzscl,block/.style={draw,circle, minimum width={width("11")+12pt},
font=\small,scale=\chipscl}]
    \foreach \x in {-2,-1,...,2} {%
      \node[anchor=north] (A\x) at (\x,0) {$\x$};
    }
    \draw (-2.2,0) -- (2.2,0);
    \draw[dashed] (-5,0) -- (-2.2,0);
     \draw[dashed] (5,0) -- (2.2,0);
    \foreach[count=\i] \a/\b in {1/1,0/1,-1/1,0/2} {%
      \node[block] at (\a,{\b*\chipscl*\chipcoef-0.5*\chipscl*\chipcoef}) {$\i$};
      }
\end{tikzpicture}
\end{center}
One can perform these firing moves until no two chips occupy the same spot. Propp conjectured that if one starts with an even number of chips at the origin, this process is  confluent and in particular the chips always end up in sorted order. For example, if we continue firing the four chips above, we necessarily will end up at:
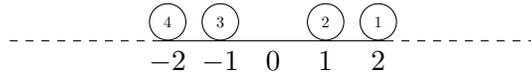
\begin{center}
\begin{tikzpicture}[scale=\chiptikzscl,block/.style={draw,circle, minimum width={width("11")+12pt},
font=\small,scale=\chipscl}]
    \foreach \x in {-2,-1,...,2} {%
      \node[anchor=north] (A\x) at (\x,0) {$\x$};
    }
    \draw (-2.2,0) -- (2.2,0);
    \draw[dashed] (-5,0) -- (-2.2,0);
     \draw[dashed] (5,0) -- (2.2,0);
    \foreach[count=\i] \a/\b in {2/1,1/1,-1/1,-2/1} {%
      \node[block] at (\a,{\b*\chipscl*\chipcoef-0.5*\chipscl*\chipcoef}) {$\i$};
      }
\end{tikzpicture}
\end{center}
It is easy to see that the labeled chip-firing process is not confluent if the initial number of chips is odd (e.g., three). Propp's sorting conjecture was recently proved by Hopkins, McConville, and Propp~\cite{hopkins2017sorting}. 

The crucial observation that motivated our present research is that we can generalize Propp's labeled chip-firing to ``other types,'' as follows. For any configuration of~$N$ labeled chips, if we define $v:=(v_1,v_2,\ldots,v_N) \in \mathbb{Z}^{N}$ by 
\[v_i:= \textrm{the position of the chip~\chip i},\] 
then, for $1\leq i < j \leq N$, we are allowed to fire chips~\chip i and~\chip j in this configuration as long as $v$ is orthogonal to $e_i-e_j$; and doing so replaces the vector $v$ by $v+(e_i-e_j)$. Here~$e_1,\dots,e_N$ are the standard basis vectors of $\Z^N$. Note that the vectors $e_i-e_j$ for~$1 \leq i < j \leq N$ are exactly (one choice of) the positive roots $\Phi^+$ of the root system~$\Phi$ of Type $A_{N-1}$. 

So let us now consider an arbitrary root system~$\Phi$ living in some Euclidean vector space $V$. Given a point $v\in V$ and a positive root $\alpha\in\Phi^+$, one is allowed to perform a \emph{central-firing move}, which consists of replacing $v$ by $v+\alpha$, whenever $v$ is orthogonal to~$\alpha$. This process generalizes Propp's labeled chip-firing moves to any (crystallographic) root system~$\Phi$. The name ``central-firing'' comes from the fact that we allow these firing moves whenever our vector belongs to a certain central hyperplane arrangement (namely, the Coxeter arrangement of $\Phi$), as opposed to other firing conditions studied in~\cite{galashin2017rootfiring1}.

For any root system $\Phi$, we say that \emph{central-firing is confluent from $v\in V$} if the process of applying central-firing moves starting from $v$ terminates and the terminal point is independent on the sequence of central-firing moves. Thus the result of~\cite{hopkins2017sorting} can be reformulated as follows.

\begin{thm}[\cite{hopkins2017sorting}]
For the root system of Type $A_{2n-1}$, central-firing is confluent from $0\in V$.
\end{thm}

Note that the classical (unlabeled) chip-firing on a line can be obtained from its labeled counterpart by forgetting the labels. In terms of the root system $\Phi$ of Type~$A_{N-1}$, this corresponds to modding out by the action of the symmetric group $S_N$, in other words, by the action of the \emph{Weyl group} $W$ of $\Phi$. One can thus generalize unlabeled chip-firing on a line to other types by extending the central-firing moves to the $W$-orbits in $V$. Given $v\in V$, an \emph{unlabeled central-firing move} consists of replacing the orbit $\W v$ with $\W(v+\alpha)$ for some $\alpha\in\Phi^+$ that is orthogonal to $v$. We say that unlabeled central-firing for $\Phi$ is \emph{confluent from $v$} if there exists $v'\in V$ such that any sequence of unlabeled central-firing moves starting at $\W v$ terminates at $\W v'$.

\begin{thm}\label{thm:UCF_confluence_intro}
  For any root system $\Phi$ and any weight $v \in P$, unlabeled central-firing is confluent from $v$.
\end{thm}

See Section~\ref{sec:background} for a definition of the weight lattice $P\subseteq V$, and  Section~\ref{sec:confl-central-firing} for a proof of Theorem~\ref{thm:UCF_confluence_intro}. In \emph{simply laced types}, unlabeled central-firing admits a simple description as a certain number game on the \emph{Dynkin diagram} of $\Phi$, and we show in this case that it has the \emph{abelian property}, just as does classical chip-firing (see e.g.~\cite[\S1.2.1]{corry2017divisors}). 

We then concentrate on the following natural question.

\begin{question}\label{question:centralconf}
  Given a root system $\Phi$ and a point $v\in V$, when is $\Phi$ confluent from $v$?
\end{question}

We will see later that for example for $\Phi=A_{2n}$ with $n\geq 1$, central-firing is not confluent from $0$; however, we conjecture that it is confluent starting from the \emph{fundamental weight} $\omega_n\in V$. In terms of chip configurations, $\omega_n$ corresponds to placing chips \chip1 through~\chip n at position~$1$ while leaving the rest of the $n+1$ chips at the origin. Thus, for example, we conjecture that the result of applying Propp's labeled chip-firing moves to the chip configuration below is independent on the firing sequence as well:
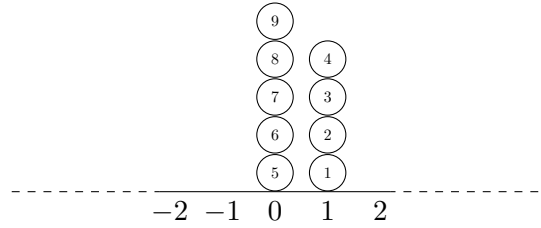
\begin{center}
\begin{tikzpicture}[scale=\chiptikzscl,block/.style={draw,circle, minimum width={width("11")+12pt},
font=\small,scale=\chipscl}]
    \foreach \x in {-2,-1,...,2} {%
      \node[anchor=north] (A\x) at (\x,0) {$\x$};
    }
    \draw (-2.2,0) -- (2.2,0);
    \draw[dashed] (-5,0) -- (-2.2,0);
     \draw[dashed] (5,0) -- (2.2,0);
    \foreach[count=\i] \a/\b in {1/1,1/2,1/3,1/4,0/1,0/2,0/3,0/4,0/5} {%
      \node[block] at (\a,{\b*\chipscl*\chipcoef-0.5*\chipscl*\chipcoef}) {$\i$};
      }
\end{tikzpicture}
\end{center}

We introduce similar chip-firing moves that correspond to root systems of other classical types (i.e., Types B, C, and D); see Section~\ref{sec:labeled-chip-firing-B-C-D}. 

Based on our computer experiments, providing even a conjectural answer to Question~\ref{question:centralconf} appears to be very hard. Rather than studying confluence of central-firing from arbitrary $v\in V$, we restrict ourselves to the case when $v$ is either the origin or a \emph{fundamental weight} for $\Phi$. Fundamental weights are certain special vectors in $V$ that correspond to the nodes of the Dynkin diagram. We denote the set of fundamental weights by $\Omega$. (See Section~\ref{sec:background} for more root system background.) We put forward a complete conjectural classification of confluence starting from the points in $\init$ (see Conjecture~\ref{conj:master_central}) and prove it in many cases. The set of weights from which central-firing is confluent seems to have a quite complicated structure in general, but (as Conjecture~\ref{conj:master_central} hints) there also appear to be interesting patterns here. In particular, to first order, confluence seems to have to do with whether the initial point is equal to the \emph{Weyl vector} $\rho\in V$ modulo the \emph{root lattice}. For example, in Types $A_{2n-1}$ and $A_{2n}$ we have respectively
\[\rho=\left(\frac{n-1}2,\frac{n-3}2,\dots,-\frac{n-1}2\right)\in\R^{2n} \quad \text{and}\quad \rho=(n,n-1,\dots,-n)\in\R^{2n+1}.\]
Thus $\rho$ is not a $\Z$-linear combination of the roots $e_i-e_j$ in $A_{2n-1}$, but in $A_{2n}$ it is. As we mentioned earlier, central-firing is confluent from $0\in V$ for $\Phi=A_{2n-1}$, but for $\Phi=A_{2n}$ it is not. This pattern seems to dictate confluence in the vast majority of cases that we consider. 

\begin{remark}
Denote by $\<\cdot,\cdot\>$ the standard inner product on $V$. Then one can make a central-firing move from $v$ to $v+\alpha$ if and only if $v$ is orthogonal to $\alpha$, i.e., if $\<v,\alpha^\vee\>=0$, where $\alpha^\vee$ is the \emph{coroot} corresponding to $\alpha$. In the first paper in this series~\cite{galashin2017rootfiring1} we showed that, for any root system, after replacing this condition by $\<v,\alpha^\vee\>=-1$ or by $\<v,\alpha^\vee\>\in\{-1,0\}$, the process becomes confluent from \emph{all} initial weights of the root system. In contrast, the condition~$\<v,\alpha^\vee\>=0$ yields a process that is confluent from some initial weights but not confluent from other initial weights, and the pattern of confluence and non-confluence seems quite complicated.
\end{remark}

Let us now give the general outline of the paper. We review background on root systems and formally define central-firing in Section~\ref{sec:background}. In Section~\ref{sec:labeled-chip-firing-B-C-D}, we interpret the central-firing moves as well as the initial configurations corresponding to the fundamental weights in terms of chips for $\Phi$ of Type A, B, C, or D.  In  Section~\ref{sec:confl-central-firing}, we describe the root system generalization of \emph{unlabeled} chip-firing on a line (obtained by considering the same process modulo the Weyl group) and prove that it is confluent from any initial configuration (Corollary~\ref{cor:unlabeled_confluent}). For simply laced types, we give an explicit combinatorial description of this process in Section~\ref{sec:unlab-centr-firing}. We also show that in this case, the unlabeled central-firing has the \emph{abelian property} (see Theorem~\ref{thm:abelian}). In Section~\ref{sec:span-central-firing}, we study the question of which weights are \emph{connected}, in the sense that central-firing starting from that weight ``spans'' the whole vector space, and apply our results to the case of unlabeled chip-firing on a line in Section~\ref{sec:interpr-terms-chips}. Finally, in Section~\ref{sec:confl-centr-firing} we give some results and conjectures regarding the confluence of central-firing starting from a point in~$\init$, including the main Conjecture~\ref{conj:master_central} that completely describes from which points in $\init$ the central-firing process is confluent. 

This paper is a sequel to~\cite{galashin2017rootfiring1}, where certain deformations of central-firing, called \emph{interval-firing processes}, were introduced and studied. This paper can be read independently from~\cite{galashin2017rootfiring1} and assumes less familiarity with the theory of root systems.

\medskip

\noindent {\bf Acknowledgements:} We thank Jim Propp, both for several useful conversations and because his introduction of labeled chip-firing and his infectious enthusiasm for exploring its properties launched this project. We also thank the anonymous referee for useful comments. The second author was supported by NSF grant~\#1122374.

\section{Background on root systems and the main definition}\label{sec:background}
In this section, we fix notation and recall a few facts from the theory of root systems.

\subsection{Root systems}
We follow the exposition given in~\cite{galashin2017rootfiring1} and we refer the reader to that paper for references for the facts that we mention.

Let us fix a real vector space $V$ of dimension~$n$ with inner product $\<\cdot,\cdot\>$. Given a nonzero vector $\alpha\in V\setminus \{0\}$, define $\alpha^\vee:=\frac{2\alpha}{\<\alpha,\alpha\>}$. The orthogonal reflection with respect to the hyperplane orthogonal to $\alpha$ is given by $s_\alpha(v):=v-\<v,\alpha^\vee\>\alpha.$
We are now ready to recall the definition of a (reduced, crystallographic) \emph{root system}.

\begin{definition}
A \emph{root system} is a finite subset $\Phi \subseteq V\setminus \{0\}$ of nonzero vectors of $V$ such that:
\begin{enumerate}
\item the vectors of $\Phi$ span $V$;
\item $s_{\alpha}(\Phi) = \Phi$ for all $\alpha \in \Phi$;
\item $(\R\cdot\alpha) \cap \Phi = \{\pm \alpha\}$ for all $\alpha \in \Phi$;
\item $\<\beta,\alpha^\vee\> \in \mathbb{Z}$ for all $\alpha,\beta \in \Phi$.
\end{enumerate} 
\end{definition}

From now on, fix a root system $\Phi$ in $V$. The vectors $\alpha$ and $\alpha^\vee$ for $\alpha\in \Phi$ are called \emph{roots} and \emph{coroots} respectively. We denote by $W$ the \emph{Weyl group} of $\Phi$, i.e., the group generated by the reflections $s_{\alpha}$ for $\alpha \in \Phi$.

We fix a set $\Delta=\{\alpha_1,\dots,\alpha_n\} \subseteq \Phi$ of \emph{simple roots}. Simple roots form a basis of~$V$ and divide the root system $\Phi = \Phi^{+} \cup \Phi^{-}$ into \emph{positive} roots $\Phi^{+}$ and \emph{negative} roots $\Phi^{-} := -\Phi^{+}$. Any positive root $\alpha \in \Phi^{+}$ is a linear combination of simple roots with nonnegative integer coefficients. Because we have fixed a set of simple roots $\Delta$, we have thus also fixed a set of positive roots $\Phi^{+}$.

The \emph{Dynkin diagram} $\DD$ of $\Phi$ is a certain graph with vertex set $[n]:=\{1,2,\dots,n\}$ and with edges defined as follows:
\begin{itemize}
\item if $\<\alpha_i,\alpha_j^\vee\>=0$ then $i$ and $j$ are not connected in $\DD$;
\item if $\<\alpha_i,\alpha_j^\vee\>=\<\alpha_j,\alpha_i^\vee\>=1$ then $i$ and $j$ are connected by one undirected edge;
\item otherwise, we have $\<\alpha_i,\alpha_j^\vee\>=1$ and $\<\alpha_j,\alpha_i^\vee\>=k$ for some $k>1$, in which case we draw $k$ directed edges from $i$ to $j$.
\end{itemize}
If all roots in $\Phi$ have the same length then we say that $\Phi$ is \emph{simply laced}. In this case, its Dynkin diagram only contains undirected edges.

We define the \emph{root lattice} $Q$ to be the set of all integer combinations of vectors in $\Phi$. The  \emph{weight lattice} $P$ is defined by
\[P := \{v\in V\colon  \<v,\alpha^\vee\> \in \mathbb{Z} \text{ for all $\alpha \in \Phi$}\}.\]
For each $i\in [n]$, define the \emph{fundamental weight} $\omega_i\in V$ by $\<\omega_i,\alpha^\vee_j\> = \delta_{i,j}$, where $\delta_{i,j}$ denotes the Kronecker delta. As we mentioned earlier, we denote the set of fundamental weights by $\Omega:=\{\omega_1,\dots,\omega_n\}$.

We say that a weight $\l$ is \emph{dominant} (resp., \emph{strictly dominant}) if $\<\l,\alpha_i^\vee\>\geq 0$ (resp., $\<\l,\alpha_i^\vee\>> 0$) for any $i\in [n]$. Thus dominant (resp., strictly dominant) weights are nonnegative (resp., positive) integer combinations of fundamental weights. The \emph{Weyl vector} $\rho$ is given by $\rho:=\sum_{i=1}^{n}\omega_i$. There is also a unique root $\theta\in\Phi^+$ called the \emph{highest root} such that, writing $\theta=\sum_{i=1}^{n}a_i\alpha_i$, the coefficients $a_i$ are maximized. If $\Phi$ is simply laced then $\theta$ is the unique root that is a dominant weight.

\def\nodescl{0.8}
\def\scll{0.3}
\begin{figure}
  \begin{tabular}{|c|c|c|}\hline
    {\begin{tikzpicture}
\node[scale=\scl,draw,circle] (1) at (-1,0) {};
\node[scale=\scl,draw,circle] (2) at (0,0) {};
\node[scale=\scl,draw,circle] (3) at (1,0) {};
\node[scale=\scl,draw,circle] (4) at (2,0) {};
\node[anchor=north,scale=\nodescl] at (1.south) {$1$};
\node[anchor=north,scale=\nodescl] at (2.south) {$2$};
\node[anchor=north,scale=\nodescl] at (3.south) {${n-1}$};
\node[anchor=north,scale=\nodescl] at (4.south) {$n$};
\draw[] (1)--(2);
\draw[,dashed] (2)--(3);
\draw[] (3)--(4);
\end{tikzpicture}} &
{\begin{tikzpicture}[decoration={markings,mark=at position 0.7 with {\arrow{>}}}]
\coordinate (1) at (-1,0) {};
\coordinate (2) at (0,0) {};
\draw[,double distance=2pt,postaction={decorate}] (2)--(1);
\node[scale=\scl,draw,circle,fill=white] (1) at (-1,0) {};
\node[scale=\scl,draw,circle,fill=white] (2) at (0,0) {};
\node[anchor=north,scale=\nodescl] at (1.south) {$1$};
\node[anchor=north,scale=\nodescl] at (2.south) {$2$};
\draw[] (1)--(2);
\end{tikzpicture}} &
{\begin{tikzpicture}[decoration={markings,mark=at position 0.7 with {\arrow{>}}}]
\node[scale=\scl,draw,circle] (1) at (-1,0) {};
\node[scale=\scl,draw,circle] (2) at (0,0) {};
\node[scale=\scl,draw,circle] (3) at (1,0) {};
\node[scale=\scl,draw,circle] (4) at (2,0) {};
\node[anchor=north,scale=\nodescl] at (1.south) {$1$};
\node[anchor=north,scale=\nodescl] at (2.south) {$2$};
\node[anchor=north,scale=\nodescl] at (3.south) {$3$};
\node[anchor=north,scale=\nodescl] at (4.south) {$4$};
\draw[] (1)--(2);
\draw[,double,postaction={decorate}] (2)--(3);
\draw[] (3)--(4);
\node[scale=\scl,draw,circle] (1) at (-1,0) {};
\node[scale=\scl,draw,circle] (2) at (0,0) {};
\node[scale=\scl,draw,circle] (3) at (1,0) {};
\node[scale=\scl,draw,circle] (4) at (2,0) {};
\end{tikzpicture}} \\
 $A_n$ & $G_2$ & $F_4$\\ \hline
{\begin{tikzpicture}[decoration={markings,mark=at position 0.7 with {\arrow{>}}}]
\coordinate (1) at (-1,0) {};
\coordinate (2) at (0,0) {};
\coordinate (3) at (1,0) {};
\coordinate (4) at (2,0) {};
\draw[] (1)--(2);
\draw[,dashed] (2)--(3);
\draw[,double,postaction={decorate}] (3)--(4);
\node[scale=\scl,draw,circle,fill=white] (1) at (-1,0) {};
\node[scale=\scl,draw,circle,fill=white] (2) at (0,0) {};
\node[scale=\scl,draw,circle,fill=white] (3) at (1,0) {};
\node[scale=\scl,draw,circle,fill=white] (4) at (2,0) {};
\node[anchor=north,scale=\nodescl] at (1.south) {$1$};
\node[anchor=north,scale=\nodescl] at (2.south) {$2$};
\node[anchor=north,scale=\nodescl] at (3.south) {${n-1}$};
\node[anchor=north,scale=\nodescl] at (4.south) {$n$};
\end{tikzpicture}} &
\multicolumn{2}{c|}{\begin{tikzpicture}
\node[scale=\scl,draw,circle] (1) at (-1,0) {};
\node[scale=\scl,draw,circle] (2) at (1,0.5) {};
\node[scale=\scl,draw,circle] (3) at (0,0) {};
\node[scale=\scl,draw,circle] (4) at (1,0) {};
\node[scale=\scl,draw,circle] (5) at (2,0) {};
\node[scale=\scl,draw,circle] (6) at (3,0) {};
\node[anchor=north,scale=\nodescl] at (1.south) {$1$};
\node[anchor=east,scale=\nodescl] at (2.west) {$2$};
\node[anchor=north,scale=\nodescl] at (3.south) {$3$};
\node[anchor=north,scale=\nodescl] at (4.south) {$4$};
\node[anchor=north,scale=\nodescl] at (5.south) {$5$};
\node[anchor=north,scale=\nodescl] at (6.south) {$6$};
\draw[] (1)--(3)--(4)--(5)--(6);
\draw[] (2)--(4);
\end{tikzpicture}} \\
  $B_n$ & \multicolumn{2}{c|}{$E_6$}  \\\hline
{\begin{tikzpicture}[decoration={markings,mark=at position 0.7 with {\arrow{>}}}]
\coordinate (1) at (-1,0) {};
\coordinate (2) at (0,0) {};
\coordinate (3) at (1,0) {};
\coordinate (4) at (2,0) {};
\draw[] (1)--(2);
\draw[,dashed] (2)--(3);
\draw[,double,postaction={decorate}] (4)--(3);
\node[scale=\scl,draw,circle,fill=white] (1) at (-1,0) {};
\node[scale=\scl,draw,circle,fill=white] (2) at (0,0) {};
\node[scale=\scl,draw,circle,fill=white] (3) at (1,0) {};
\node[scale=\scl,draw,circle,fill=white] (4) at (2,0) {};
\node[anchor=north,scale=\nodescl] at (1.south) {$1$};
\node[anchor=north,scale=\nodescl] at (2.south) {$2$};
\node[anchor=north,scale=\nodescl] at (3.south) {${n-1}$};
\node[anchor=north,scale=\nodescl] at (4.south) {$n$};
\end{tikzpicture}} &
\multicolumn{2}{c|}{\begin{tikzpicture}
\node[scale=\scl,draw,circle] (1) at (-1,0) {};
\node[scale=\scl,draw,circle] (2) at (1,0.5) {};
\node[scale=\scl,draw,circle] (3) at (0,0) {};
\node[scale=\scl,draw,circle] (4) at (1,0) {};
\node[scale=\scl,draw,circle] (5) at (2,0) {};
\node[scale=\scl,draw,circle] (6) at (3,0) {};
\node[scale=\scl,draw,circle] (7) at (4,0) {};
\node[anchor=north,scale=\nodescl] at (1.south) {$1$};
\node[anchor=east,scale=\nodescl] at (2.west) {$2$};
\node[anchor=north,scale=\nodescl] at (3.south) {$3$};
\node[anchor=north,scale=\nodescl] at (4.south) {$4$};
\node[anchor=north,scale=\nodescl] at (5.south) {$5$};
\node[anchor=north,scale=\nodescl] at (6.south) {$6$};
\node[anchor=north,scale=\nodescl] at (7.south) {$7$};
\draw[] (1)--(3)--(4)--(5)--(6)--(7);
\draw[] (2)--(4);
\end{tikzpicture}} \\
  $C_n$ & \multicolumn{2}{c|}{$E_7$}  \\\hline
{\begin{tikzpicture}
\node[scale=\scl,draw,circle] (1) at (-1,0) {};
\node[scale=\scl,draw,circle] (2) at (0,0) {};
\node[scale=\scl,draw,circle] (3) at (1,0) {};
\node[scale=\scl,draw,circle] (4) at (2,-0.3) {};
\node[scale=\scl,draw,circle] (5) at (2,0.3) {};
\node[anchor=north,scale=\nodescl] at (1.south) {$1$};
\node[anchor=north,scale=\nodescl] at (2.south) {$2$};
\node[anchor=north,scale=\nodescl] at (3.south) {${n-2}$};
\node[anchor=north,scale=\nodescl] at (4.south) {${n-1}$};
\node[anchor=south,scale=\nodescl] at (5.north) {$n$};
\draw[] (1)--(2);
\draw[,dashed] (2)--(3);
\draw[] (3)--(4);
\draw[] (3)--(5);
\end{tikzpicture}} &
\multicolumn{2}{c|}{\begin{tikzpicture}
\node[scale=\scl,draw,circle] (1) at (-1,0) {};
\node[scale=\scl,draw,circle] (2) at (1,0.5) {};
\node[scale=\scl,draw,circle] (3) at (0,0) {};
\node[scale=\scl,draw,circle] (4) at (1,0) {};
\node[scale=\scl,draw,circle] (5) at (2,0) {};
\node[scale=\scl,draw,circle] (6) at (3,0) {};
\node[scale=\scl,draw,circle] (7) at (4,0) {};
\node[scale=\scl,draw,circle] (8) at (5,0) {};
\node[anchor=north,scale=\nodescl] at (1.south) {$1$};
\node[anchor=east,scale=\nodescl] at (2.west) {$2$};
\node[anchor=north,scale=\nodescl] at (3.south) {$3$};
\node[anchor=north,scale=\nodescl] at (4.south) {$4$};
\node[anchor=north,scale=\nodescl] at (5.south) {$5$};
\node[anchor=north,scale=\nodescl] at (6.south) {$6$};
\node[anchor=north,scale=\nodescl] at (7.south) {$7$};
\node[anchor=north,scale=\nodescl] at (8.south) {$8$};
\draw[] (1)--(3)--(4)--(5)--(6)--(7)--(8);
\draw[] (2)--(4);
\end{tikzpicture}} \\
  $D_n$ & \multicolumn{2}{c|}{$E_8$}  \\\hline
  \end{tabular}
\caption{Dynkin diagrams of all irreducible root systems.} \label{fig:dynkinclassification}
\end{figure}
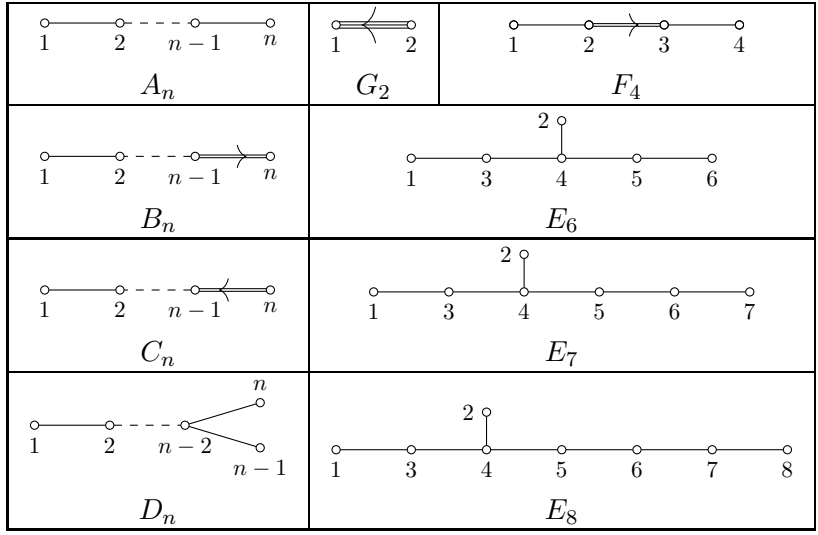

We say that $\Phi$ is \emph{irreducible} if its Dynkin diagram is connected. From now on in the paper we assume that $\Phi$ is an irreducible root system. Dynkin diagrams of all irreducible root systems are shown in Figure~\ref{fig:dynkinclassification}. The simple roots are numbered as in~\cite{bourbaki2002lie}.

For a subspace $H\subseteq V$ spanned by some roots of $\Phi$, $\Phi\cap H$ is another root system which we call a \emph{sub-root system of $\Phi$}. Given a subset $I\subseteq [n]$, we denote by $\Phi_{I}\subseteq \Phi$ the \emph{parabolic sub-root system} of $\Phi$ defined by $\Phi_{I}:=\Phi\cap \Span\{\alpha_i\colon i\in I\}$. We use $W_{I}$ to denote the corresponding \emph{parabolic subgroup} of $W$: this is the subgroup generated by~$s_{\alpha}$ for $\alpha \in \Phi_{I}$. For a dominant weight $\l\in P$, we denote $I_\l^0:=\{i\in [n]\colon \<\l,\alpha_i^\vee\>=0\}$.

Given a weight $\l\in P$, we denote by $\Pi(\l):=\mathrm{ConvexHull}(\{w\l\colon w\in W\})$ the \emph{permutohedron} associated to~$\l$. We let $\Pi^Q(\l):=\{\m\in\Pi(\l)\colon \l-\m\in Q\}$ denote its set of \emph{lattice points}. We say that a nonzero dominant weight $\l\in P$ is \emph{minuscule} if $\Pi^Q(\l)$ consists only of the vertices of $\Pi(\l)$, i.e., of the $W$-orbit of $\l$. We let $\Omega_m$ be the set of minuscule weights. It is known that $\Omega_m\subseteq\Omega$ and that every class in $P/Q$ contains a unique element from $\Omega^0_m := \Omega_m\cup\{0\}$.

\subsection{Root systems of classical type}\label{sec:root-syst-class}
Let us use the notation $\Z+\frac12:=\{a+\frac12\colon a\in\Z\}$. For $\Phi$ of \emph{classical type}, i.e., of Type A, B, C, or D, we use the following explicit realizations of $\Phi$.

\begin{itemize}
\item[$A_n$:] We let $V:=\R^{n+1}/\<(1,1,\dots,1)\>$ be the $n$-dimensional space of vectors in $\R^{n+1}$ considered modulo adding a multiple of the vector $(1,1,\dots,1)$. We often identify this space with the subspace of $\R^{n+1}$ where the sum of coordinates is equal to $0$. In particular, the inner product $\<\cdot,\cdot\>$ on $V$ comes from the identification with this subspace of $\R^{n+1}$. The weight lattice $P\subseteq V$ is given by $\Z^{n+1}/\<(1,1,\dots,1)\>$.  The roots of $\Phi$ are given by $e_i-e_j$ for all $i\neq j\in[n+1]$. The positive roots are the ones with $i<j$, and the simple roots are $\alpha_i:=e_i-e_{i+1}$ for $i\in [n]$.
\item[$B_n$:] We let $V:=\R^n$ and $P:=\Z^n\cup (\Z+\frac12)^n$. The roots of $\Phi$ are given by $\pm e_i\pm e_j$ for all $i< j\in[n]$ and $\pm e_i$ for all $i\in [n]$. The positive roots are $ e_i\pm e_j$ for all $i< j\in[n]$ and $e_i$ for all $i\in [n]$. The simple roots are $\alpha_i:=e_i-e_{i+1}$ for $1\leq i<n$ and $\alpha_n:=e_n$.
\item[$C_n$:] We let $V:=\R^n$ and $P:=\Z^n$. The roots of $\Phi$ are given by $\pm e_i\pm e_j$ for all $i< j\in[n]$ and $\pm 2e_i$ for all $i\in [n]$. The positive roots are $ e_i\pm e_j$ for all $i< j\in[n]$ and $2e_i$ for all $i\in [n]$. The simple roots are $\alpha_i:=e_i-e_{i+1}$ for $1\leq i<n$ and $\alpha_n:=2e_n$.
\item[$D_n$:] We let $V:=\R^n$ and $P:=\Z^n\cup (\Z+\frac12)^n$. The roots of $\Phi$ are given by $\pm e_i\pm e_j$ for all $i< j\in[n]$. The positive roots are $ e_i\pm e_j$ for all $i< j\in[n]$. The simple roots are $\alpha_i:=e_i-e_{i+1}$ for $1\leq i<n$ and $\alpha_n:=e_{n-1}+e_n$.
\end{itemize}

\subsection{Main definition}\label{sec:main_def}
We define \emph{central-firing} to be a binary relation $\raPP$ on $P$: for a weight~$\lambda \in P$, we have that~$\l\raPP \l+\alpha$ whenever $\alpha\in\Phi^+$ is such that $\<\l,\alpha^\vee\>=0$. 

Given a binary relation $\ra$ on a set $X$, we denote by $\raAst$ its reflexive, transitive closure and we say that $\ra$ is \emph{confluent from~$x\in X$} if for any $y,y'\in X$ such that~$x\raAst y$ and~$x\raAst y'$, there exists $z\in X$ such that~$y\raAst z$ and~$y'\raAst z$. We say that $\ra$ is \emph{confluent} if it is confluent from any $x\in X$, and we call it \emph{terminating} if there exists no infinite sequence $x_0,x_1,\dots \in X$ such that $x_i\ra x_{i+1}$ for all $i\geq 0$. We say $x \in X$ is \emph{$\ra$-stable} if there is no $y \in X$ with $x\ra y$. If $\ra$ is confluent from $x \in X$ and is terminating, then there is a unique stable $y \in X$ with $x \raAst y$ called the \emph{$\ra$-stabilization} of $x$. The proof of the following result is analogous to that of~\cite[\centraltermination]{galashin2017rootfiring1}.

\begin{prop}\label{prop:centraltermination}
  For any root system~$\Phi$, the relation $\raPP$ is terminating.
\end{prop}
\begin{proof}
For~$\lambda \in P$, consider the function $\phi(\lambda) :=  \<2\rho-\lambda,  2\rho-\lambda\>$. Suppose that $\lambda\raPP \lambda+\alpha$ for some $\alpha\in\Phi^+$. Thus $\<\alpha,\lambda\>=0$; together with $\<2\rho,\alpha\>=\<2\rho,\frac{\<\alpha,\alpha\>}2\alpha^\vee\>\geq \<\alpha,\alpha\>$, we get that
\[\phi(\lambda) - \phi(\lambda+\alpha) = \< 2\rho-\lambda, 2\rho-\lambda\>- \< 2\rho-(\lambda+\alpha), 2\rho-(\lambda+\alpha)\>=\< 4\rho-\alpha,\alpha\>\geq \<\alpha,\alpha\>.\]
Thus after each firing move, $\phi(\lambda)$ decreases by at least $\min_{\alpha\in\Phi^+}\<\alpha,\alpha\>>0$, and since the quantity $\phi(\lambda)$ is manifestly nonnegative, we see that $\raPP$ is terminating.
\end{proof}

\section{Labeled chip-firing for classical types}\label{sec:labeled-chip-firing-B-C-D}

We will consider configurations of chips on $\Z$ and various chip-firing moves between them. Let us introduce the moves that will describe the relations $\raPP$ for all~$\Phi$ of classical type.

\begin{definition}
As explained in the introduction, a configuration of $N$ chips on the infinite path graph~$\mathbb{Z}$ corresponds to a vector $v=(v_1,\ldots,v_N)\in \Z^N$ where $v_i$ is the position of \chip i. We also want to sometimes consider vectors $v\in(\Z+\frac12)^N$, which we think of as configurations of $N$ labeled chips on the graph~$\Z+\frac12$ which is isomorphic to $\Z$ but has vertex labels shifted by $\frac12$. Given a configuration $v\in\Z^N$ or $v\in(\Z+\frac12)^N$, we define the following four types of \emph{moves}:
  \begin{enumerate}[(a)]
  \item\label{move:A} for $i<j$, if chips~\chip i and \chip j are in the same position (i.e., $v_i=v_j$), move chip~\chip i one step to the right (i.e., increase $v_i$ by $1$) and chip~\chip j one step to the left (i.e., decrease $v_j$ by $1$);
  \item\label{move:B} for $i\in[N]$, if chip~\chip i is at the origin (i.e., $v_i=0$), move it one step to the right;
  \item\label{move:C} for $i\in[N]$, if chip~\chip i is at the origin (i.e., $v_i=0$), move it two steps to the right;
  \item\label{move:D} for $i<j$, if chips~\chip i and \chip j are in the opposite positions (i.e., $v_i=-v_j$), move both chips one step to the right.
  \end{enumerate}
\end{definition}

The following interpretation is clear from the explicit constructions realizing the corresponding root system in $\R^N$ given in Section~\ref{sec:root-syst-class}.
\begin{prop}
  Two chip configurations $u,v\in\Z^N$ satisfy $u\raPPAst v$ if and only if $v$ can be obtained from $u$ by applying 
  \begin{itemize}
  \item the moves~\eqref{move:A}, if $\Phi$ is of Type $A_{N-1}$;
  \item the moves~\eqref{move:A},~\eqref{move:B}, and~\eqref{move:D}, if $\Phi$ is of Type $B_{N}$;
  \item the moves~\eqref{move:A},~\eqref{move:C}, and~\eqref{move:D}, if $\Phi$ is of Type $C_{N}$;
  \item the moves~\eqref{move:A} and~\eqref{move:D}, if $\Phi$ is of Type $D_{N}$.
  \end{itemize}
  Note that for $\Phi$ of Type $A_{N-1}$ we need to consider $u$ and $v$ modulo $\<(1,1,\ldots,1)\>$ and observe that the moves~\eqref{move:A} are still well-defined modulo $\<(1,1,\ldots,1)\>$.
\end{prop}

Let us now also describe the initial configurations that correspond to the weights in the set $\init$. For each $\Phi$ of classical type, the zero weight corresponds to the configuration of $N$ chips at the origin. For $\Phi$ of Types $A_{N-1}$, $B_N$, or $C_N$, the fundamental weight $\omega_i$, $1\leq i\leq N-1$, corresponds to the configuration of the first $i$ chips at position~$1$ and the last $N-i$ chips at the origin. For $\Phi$ of Type $B_N$, the weight $\omega_N$ corresponds to all chips being at position $\frac12$. For $\Phi$ of Type $C_N$, the weight $\omega_N$ corresponds to all chips being at position $1$. Finally, for $\Phi$ of Type $D_N$, the fundamental weight $\omega_i$, $1\leq i\leq N-2$ corresponds to the configuration of the first $i$ chips being at position $1$ and the remaining $N-i$ chips being at the origin, $\omega_N$ corresponds to all chips being at position $\frac12$, and $\omega_{N-1}$ differs from $\omega_n$ only in the position of chip~\chip N which is at position $-\frac12$. See Figure~\ref{fig:initial_confs} for an illustration. We note that for each initial configuration $v$ in Figure~\ref{fig:initial_confs}, central-firing for each of the listed root systems is confluent from $v$ (see Remark~\ref{rmk:confluence_classical}).

\begin{remark}
For $\Phi=A_{N-1}$ the weight lattice is $P=\mathbb{Z}^N/\<(1,1,\dots,1)\>$ and not~$\mathbb{Z}^N$; however, the central-firing process can be lifted in an obvious way to all of $\mathbb{Z}^N$ and in this way we precisely recover Propp's original labeled chip-firing on~$\mathbb{Z}$. The fact that we can mod out by $(1,1,\dots,1)$ is reflected in the fact that the labeled chip-firing process is unchanged if we translate all chips to the left or to the right by the same amount.
\end{remark}

\begin{remark}\label{rmk:half_integers}
When the coordinates of chips are half-integers, one can never perform moves~\eqref{move:B} and~\eqref{move:C}. Thus for example the central-firing processes for $\Phi$ of Type $B_N$ or $D_N$ starting from $\omega_N$ are identical. We shall later see that they are conjecturally both confluent for each $N$.
\end{remark}

 We believe that this chip interpretation will help prove some parts of Conjecture~\ref{conj:master_central} below, because it allows chip-firing arguments similar to those used for the usual (i.e., Type A) labeled chip-firing in~\cite{hopkins2017sorting} to be applied to the other types as well.
 
\begin{figure}

\def\chiptikzscl{0.7}
  \def\chipscl{0.6}
  \def\chipcoef{1.2}
\def\dashed{2}
  
\begin{tabular}{|c|c|c|c|}\hline

\begin{tikzpicture}[scale=\chiptikzscl,block/.style={draw,circle, minimum width={width("11")+12pt},
font=\small,scale=\chipscl}]
    \foreach \x in {-1,...,1} {%
      \node[anchor=north] (A\x) at (\x,0) {$\x$};
    }
    \draw[dashed] (-\dashed,0) -- (\dashed,0);
    \draw (-1.2,0) -- (1.2,0);
    \foreach[count=\i] \a/\b in {1/1,1/2,1/3,0/1,0/2,0/3,0/4} {%
      \node[block] at (\a,{\b*\chipscl*\chipcoef-0.5*\chipscl*\chipcoef}) {$\i$};
      }
\end{tikzpicture} &
\begin{tikzpicture}[scale=\chiptikzscl,block/.style={draw,circle, minimum width={width("11")+12pt},
font=\small,scale=\chipscl}]
      \node[anchor=north] (A1) at (0.5,0) {$\frac12$};
      \node[anchor=north] (Am1) at (-0.5,0) {$-\frac12$};
    \draw[dashed] (-\dashed,0) -- (\dashed,0);
    \draw (-1.2,0) -- (1.2,0);
    \foreach[count=\i] \a/\b in {0.5/1,0.5/2,0.5/3,0.5/4,0.5/5} {%
      \node[block] at (\a,{\b*\chipscl*\chipcoef-0.5*\chipscl*\chipcoef}) {$\i$};
      }
\end{tikzpicture} &
\begin{tikzpicture}[scale=\chiptikzscl,block/.style={draw,circle, minimum width={width("11")+12pt},
font=\small,scale=\chipscl}]
    \foreach \x in {-1,...,1} {%
      \node[anchor=north] (A\x) at (\x,0) {$\x$};
    }
    \draw[dashed] (-\dashed,0) -- (\dashed,0);
    \draw (-1.2,0) -- (1.2,0);
    \foreach[count=\i] \a/\b in {1/1,1/2,1/3,1/4,1/5,1/6} {%
      \node[block] at (\a,{\b*\chipscl*\chipcoef-0.5*\chipscl*\chipcoef}) {$\i$};
      }
      \node at (0,4.4) {\,}; 
\end{tikzpicture} &
\begin{tikzpicture}[scale=\chiptikzscl,block/.style={draw,circle, minimum width={width("11")+12pt},
font=\small,scale=\chipscl}]
      \node[anchor=north] (A1) at (0.5,0) {$\frac12$};
      \node[anchor=north] (Am1) at (-0.5,0) {$-\frac12$};
    \draw[dashed] (-\dashed,0) -- (\dashed,0);
    \draw (-1.2,0) -- (1.2,0);
    \foreach[count=\i] \a/\b in {0.5/1,0.5/2,0.5/3,0.5/4,-0.5/1} {%
      \node[block] at (\a,{\b*\chipscl*\chipcoef-0.5*\chipscl*\chipcoef}) {$\i$};
      }
    \end{tikzpicture}  \\
  $\omega_3$ for $\Phi$ of & $\omega_5$ for $\Phi$ of & $\omega_6$ for $\Phi$ of & $\omega_4$ for $\Phi$ of \\
    Type~$A_6$, $B_7$, or $C_7$ & Type~$B_5$ or $D_5$ & Type~$C_6$ & Type~$D_5$\\\hline
\end{tabular}

  \caption{\label{fig:initial_confs} Examples of initial chip configurations corresponding to some weights in $\init$.}
\end{figure}
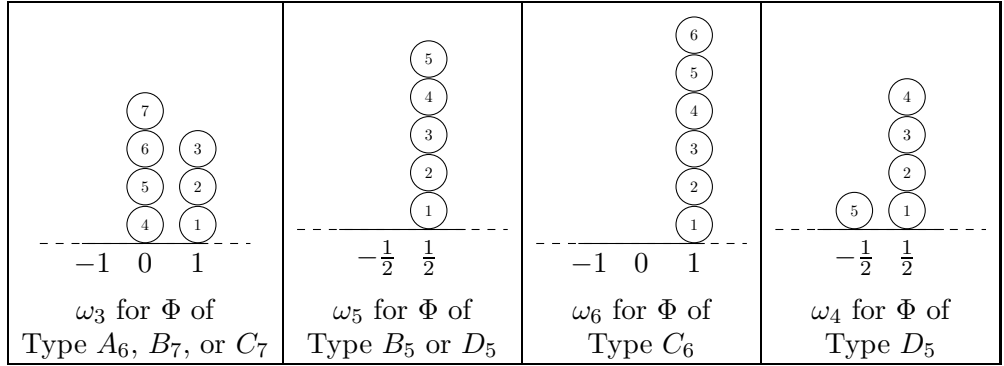

\section{Confluence of central-firing modulo the Weyl group}\label{sec:confl-central-firing}
 
Let $X$ be a set, $\ra$ a binary relation, and $G$ a group acting on $X$. For $x \in X$, we write $\G x$ to denote the orbit of $x$ under $G$, and we write $X/G$ for the set of orbits of $X$ under $G$. The relation $\ra$ descends to a relation, also denoted $\ra$, on $X/G$ as follows: we have $\G x \ra \G y$ if and only if there exists $x' \in \G x$ and $y' \in \G y$ such that~$x' \ra y'$. Note that the notation  $\G x \raAst \G y$ is inherently ambiguous because it is not clear if it means that we mod out by the group action before or after taking the reflexive transitive closure. In what follows will take $\G x \raAst \G y$ to mean that there exists $t \geq 0$ and $x_0,x_1,\ldots,x_t \in X$ such that 
\[\G x = \G{x_0} \ra \G{x_1} \ra \cdots \ra \G{x_t}= \G y.\]
However, in the case that we care about, central-firing modulo the Weyl group, this ambiguity is actually irrelevant and the two possible interpretations coincide as the next proposition shows. Of course, when $\Phi=A_{N-1}$, the relation $\raPP$ on $P/W$ corresponds exactly to unlabeled chip-firing of $N$ chips on a line. We refer to the relation $\raPP$ on~$P/W$ as \emph{unlabeled central-firing}.

\begin{prop}\label{prop:rapp_lift}
For $\lambda, \mu \in P$, we have $\W\lambda \raPP \W\mu$ if and only if there is $\mu' \in \W\mu$ with $\lambda  \raPP \mu'$.
\end{prop}
\begin{proof}
Let $w\lambda \in\W\lambda$ be such that $w\lambda\raPP w\lambda+\alpha\in\W\mu$ for some $\alpha\in\Phi^+$ which satisfies~$\<w\lambda,\alpha^\vee\>=0$. Since $w$ is an orthogonal transformation, $\<\lambda,w^{-1}(\alpha)^\vee\>=0$ as well. If $w^{-1}(\alpha)\in\Phi^+$, then we are done since we found a firing move $\lambda\raPP \lambda+w^{-1}(\alpha)$ with $\lambda+w^{-1}(\alpha)\in\W\mu$. If $w^{-1}(\alpha) \in \Phi^{-}$, let $\mu':=s_{w^{-1}(\alpha)}(\lambda+w^{-1}(\alpha))\in\W\mu$. Since we have~$\<\lambda,w^{-1}(\alpha)^\vee\>=0$, it follows that $\mu'=\lambda-w^{-1}(\alpha)$ and now $-w^{-1}(\alpha)$ is a positive root, so we are done. 
\end{proof}

\begin{cor}
For $\lambda, \mu \in P$, we have $\W\lambda \raPPAst \W\mu$ if and only if there is $\mu' \in \W\mu$ with $\lambda  \raPPAst \mu'$.
\end{cor}

Similarly to the definition of $\raPP$ in Section~\ref{sec:main_def}, let us define another binary relation~$\raPoint$ on $P$ (\emph{``central-firing of all the roots''}) by $\l\raPoint \l+\alpha$ whenever $\lambda \in P$ and $\alpha\in\Phi$ is such that $\<\l,\alpha^\vee\>=0$. Surprisingly, central-firing of the positive roots and central-firing of all the roots determine the same relation on $P/W$:

\begin{prop}
For $\lambda, \mu \in P$, we have $\W\lambda \raPP \W\mu$ if and only if $\W\lambda \raPoint \W\mu$.
\end{prop}
\begin{proof}
It suffices to show that if $\W\lambda \raPoint\W\mu$ then $\W\lambda \raPP\W\mu$. Indeed, suppose that we have~$\lambda \raPoint\mu$ for some $\lambda,\mu\in P$. Then $\mu=\lambda+\alpha$ for some $\alpha\in\Phi$. If $\alpha\in\Phi^+$ then clearly~$\W\lambda\raPP\W\mu$ and we are done. If $\alpha\in\Phi^-$, then set $\mu':=\lambda-\alpha=s_\alpha(\mu)\in\W\mu$. We then have $\lambda\raPP\mu'$.
\end{proof}

\begin{prop}
The relation~$\raPP$ on $P/W$ is terminating.
\end{prop}
\begin{proof}
Suppose that there exists an infinite path $\W\lambda_1\raPP\W\lambda_2\raPP\dots$. Then by Proposition~\ref{prop:rapp_lift}, there exists $\mu_2\in\W\lambda_2$ such that $\lambda_1\raPP\mu_2$. By Proposition~\ref{prop:rapp_lift} again, there exists $\mu_3\in\W\lambda_3$ such that $\mu_2\raPP\mu_3$, and so on. We obtain an infinite sequence~$\lambda_1\raPP\mu_2\raPP\mu_3\raPP\dots$ which contradicts Proposition~\ref{prop:centraltermination}.
\end{proof}

Now we proceed to prove that unlabeled central-firing is confluent (from every initial orbit $\W\l$). In order to do so, we will use \emph{Newman's lemma}, a.k.a., the \emph{diamond lemma}~\cite{newman1942theories}, which we now explain.

\begin{definition}
We say that a relation $\ra$ on a set $X$ is \emph{locally confluent} if for any~$x,y,y'\in X$ with $x\ra y$ and $x\ra y'$, there exists $z\in X$ such that $y\raAst z$ and $y'\raAst z$.
\end{definition}

\begin{lemma}[\cite{newman1942theories}]
  Let $\ra$ be a terminating relation on $X$. Then $\ra$ is confluent if and only if $\ra$ is locally confluent.
\end{lemma}

\begin{lemma}
The relation~$\raPP$ on $P/W$ is locally confluent.
\end{lemma}
\begin{proof}
Let $\l,\m,\m'\in P$ be such that $\W\l\raPP \W\m$ and $\W\l\raPP\W\m'$. By Proposition~\ref{prop:rapp_lift}, we may choose $\m$ and $\m'$ so that $\l\raPP \m$ and $\l\raPP\m'$. Let $\alpha:=\m-\l$ and $\beta:=\m'-\l$. Thus~$\alpha$ and $\beta$ are positive roots that are both orthogonal to $\l$. We may assume that $\alpha\neq\beta$. Consider now the affine $2$-dimensional plane $H$ spanned by $\alpha$ and $\beta$ that passes through~$\l$. If we can show that there exists $\nu\in H$ such that $\W\m\raPP\W\nu$ and~$\W\m'\raPP\W\nu$ then we are done with the proof. Therefore it is enough to show that for the sub-root system~$\Phi'$ of $\Phi$ spanned by $\alpha$ and $\beta$, the relation $\raPPprime$ on $P'/W'$ is confluent, where $P'$ and $W'$ denote the weight lattice and the Weyl group of $\Phi'$. 

Thus we can now assume that $\Phi=\Phi'$ is a rank $2$ root system. Note, in rank~$2$, that to establish confluence we only need to check confluence from $\W0$ (because there is at most one firing move from any other orbit). This is easily verified by hand in each of the four possible cases: $A_1\oplus A_1, A_2, B_2, G_2$. We need to check that for any $\beta_1,\beta_2\in\Phi^+$, there exists $\lambda \in P$ such that $\W\beta_1\raPPAst\W\lambda$ and $\W\beta_2\raPPAst\W\lambda$. For $A_1\oplus A_1$ this is trivial, so we can assume $\Phi$ is irreducible. Then, if $\beta_1$ and $\beta_2$ have the same length we get $\W\beta_1=\W\beta_2$ and so there is nothing to check. Thus we can assume that $\Phi$ is not simply laced and $\beta_1$ is short and $\beta_2$ is long. Since the answer only depends on $\W\beta_1$ and $\W\beta_2$, we are free to choose any short $\beta_1$ and long $\beta_2$. So for $\Phi=B_2$ we can take $\beta_1=\alpha_2$ and $\beta_2=\alpha_1+2\alpha_2$ (with the numbering of the simple roots as in Figure~\ref{fig:dynkinclassification}) and $\lambda=\beta_2$, since then $\<\beta_1,\alpha_1+\alpha_2\> = 0$ and $\alpha_1+\alpha_2\in \Phi^{+}$. And for $\Phi=G_2$ we can take $\beta_1=\alpha_1$ and $\beta_2=3\alpha_1+2\alpha_2$ and $\lambda=\beta_1+\beta_2$, since then $\<\beta_1,\beta_2\>=0$.
\end{proof}

\begin{cor}\label{cor:unlabeled_confluent}
The relation~$\raPP$ on $P/W$ is confluent (and terminating).
\end{cor}

\begin{remark}
Unlabeled central-firing is a generalization of classical chip-firing to other root systems $\Phi$. Another such generalization, studied in detail by Benkart, Klivans, and Reiner~\cite{benkart2016chip}, is $M$-matrix chip-firing with respect to the Cartan matrix~$\mathbf C$ of $\Phi$. Such Cartan matrix chip-firing is also confluent for all root systems, starting with any initial configuration. We note that these generalizations are somewhat ``orthogonal'' to each other: for example, in Type~$A_{N-1}$, unlabeled central-firing corresponds to chip-firing of~$N$ chips on the infinite path graph; whereas the Cartan matrix chip-firing corresponds to chip-firing of any number of chips on the cycle graph with~$N$ vertices. For a more direct connection between the results of~\cite{benkart2016chip} and our work, see~\cite[Remark~10.3]{galashin2017rootfiring1}.
\end{remark}

Corollary~\ref{cor:unlabeled_confluent} says that to decide if central-firing is confluent from $\lambda$, i.e., to answer Question~\ref{question:centralconf}, we only need to verify that there is a unique chamber which every central-firing sequence from $\lambda$ terminates in. However, in practice this does not necessarily help that much to resolve Question~\ref{question:centralconf}; e.g., the main difficulty in the analysis of labeled chip-firing in~\cite{hopkins2017sorting} was precisely to show that the labeled chip-firing process \emph{sorts} the chips (from the appropriate initial configuration).

In many cases we can say exactly what the stabilization of $\W\lambda$ is. 

\begin{prop}\label{prop:unlabeled_stabilization}
Suppose that $\lambda \in \Pi^{Q}(\rho+\omega)$ for some $\omega \in \Omega^0_m$. Then $\W(\rho+\omega)$ is the $\raPP$-stabilization of $\W\lambda$.
\end{prop}
\begin{proof}
  It is clear that $\W(\rho+\omega)$ is $\raPP$-stable since $\rho+\omega$ is strictly dominant. In fact, we claim that $\rho+\omega$ is the only strictly dominant weight in $\Pi^{Q}(\rho+\omega)$. Indeed, suppose that~$\nu$ is strictly dominant and belongs to~$\Pi^{Q}(\rho+\omega)$. Since it is strictly dominant, we have~$\nu= \rho+\mu$ for some dominant weight $\mu$. Recall the following well-known fact whose proof is given in~\cite[\permcontainment]{galashin2017rootfiring1} (see also~\cite{stembridge1998partial}).
  \begin{lemma}\label{lemma:permcontainment}
    For two dominant weights $\m,\m'\in P$, we have $\m\in \Pi^Q(\m')$ if and only if~$\m'-\m$ is a linear combination of simple roots with nonnegative integer coefficients. 
  \end{lemma}
 Thus $(\rho+\omega)-\nu$ is an integer combination of simple roots with nonnegative coefficients. Therefore the same is true for $\omega-\mu$, and hence $\mu \in \Pi^{Q}(\omega)$ again by Lemma~\ref{lemma:permcontainment}. By definition, this forces~$\mu=\omega$ and thus~$\nu= \rho+\omega$.

So the vertices of $\Pi(\rho+\omega)$ are the only weights in $\Pi^{Q}(\rho+\omega)$ that are~$\raPP$-stable. Let us now recall another result that follows from~\cite[\permtrap]{galashin2017rootfiring1}.
 \begin{lemma}\label{lem:permtrap}
   If $\m\in\Pi^Q(\rho+\m'')$ for dominant weights $\m$ and $\m''$ then $\m'\in \Pi^Q(\rho+\m'')$ for any $\m'\in P$ such that $\m\raPP\m'$.
 \end{lemma}
By Lemma~\ref{lem:permtrap} together with Proposition~\ref{prop:centraltermination} we know that any central-firing sequence starting at a weight in $\Pi^{Q}(\rho+\omega)$ must terminate at a weight in $\Pi^{Q}(\rho+\omega)$. So such a firing sequence must terminate at a vertex of~$\Pi(\rho+\omega)$. Thus indeed we have~$\W\lambda \raPPAst \W(\rho+\omega)$.
\end{proof}

\section{Unlabeled central-firing on simply laced Dynkin diagrams}\label{sec:unlab-centr-firing}

For $\Phi$ of classical type, the moves from Section~\ref{sec:labeled-chip-firing-B-C-D} allow one to give a similar description of unlabeled central-firing in these types. For example, for Type A, forgetting the labels of the chips yields exactly the unlabeled central-firing process. In this section, we give a very different description of the same process. It turns out that when $\Phi$ is simply laced, unlabeled central-firing can be reformulated as a certain number game with simple rules on the Dynkin diagram $\DD$ of $\Phi$. The goal of this section is to describe these rules and generalize the \emph{abelian} property of classical chip-firing, which says that firing moves always ``commute.''

\begin{prop}\label{prop:simply_laced_dominant}
Suppose $\Phi$ is simply laced. Let $\lambda \in P$ be a dominant weight. Then if $\W\lambda \raPP \W\mu$, there is a dominant $\mu' \in \W\mu$ such that $\lambda \raPP \mu'$.
\end{prop}
\begin{proof}
  Let $\lambda \in P$ be dominant and suppose that $\W\lambda \raPP \W\mu$. By Proposition~\ref{prop:rapp_lift}, we may assume that $\lambda \raPP \mu$ so let $\beta\in\Phi^+$ be such that $\mu=\lambda+\beta$. Since $\<\lambda,\beta^\vee\> = 0$, we have $\beta \in \Phi_{I^0_{\lambda}}$. Let $\Phi'\subseteq \Phi_{I^0_{\lambda}}$  be the irreducible sub-root system of $\Phi_{I^0_{\lambda}}$ that contains $\beta$. 
  
  Let $\theta'$ be the highest root of $\Phi'$. We claim that $\lambda+\theta'$ is a dominant weight that belongs to $\W\mu$. First note that since $\Phi$ is simply laced and $\Phi'$ is irreducible, $\theta'$ can be obtained from $\beta$ by the action of the Weyl group $W'$ of $\Phi'$ (which stabilizes $\lambda$), and thus $\lambda+\theta'\in\W\mu$. Second, let us show that $\lambda+\theta'$ is dominant. For any simple root~$\alpha_i$, we have
  \[\<\lambda+\theta',\alpha^\vee_i\>=\<\lambda,\alpha^\vee_i\>+\<\theta',\alpha^\vee_i\>.\]
  Suppose the first term $\<\lambda,\alpha^\vee_i\>$ in the right hand side is nonzero; then it must be positive. The second term $\<\theta',\alpha^\vee_i\>$ is greater than or equal to $-1$ because $\Phi$ is simply laced. Therefore, their sum is nonnegative. Suppose now that $\<\lambda,\alpha^\vee_i\>$ is zero. Then $\alpha_i$ is a simple root of $\Phi_{I^0_{\lambda}}$ and hence $\<\theta',\alpha^\vee_i\>\geq0$. 
\end{proof}

\begin{remark} \label{rem:nonsimplylaced_central}
Proposition~\ref{prop:simply_laced_dominant} does not hold in general when $\Phi$ is not simply laced. This is already apparent for $\Phi=B_2$ and $\Phi=G_2$ when starting from the fundamental weight corresponding to the long simple root.
\end{remark}

Let us now explicitly describe the relation $\raPP$ on $P/W$ for simply laced root systems. To do so, we need to discuss affine Dynkin diagrams. Associated to every connected, simply laced Dynkin diagram $\DD$ with vertex set $[n]$ is a (unique) \emph{affine Dynkin diagram}, denoted~$\DDaff$, with vertex set $[n]\cup\{0\}$ and which contains $\DD$ as a subgraph. These affine Dynkin diagrams are depicted in Figure~\ref{fig:affinedds}. See~\cite[VI, \S3]{bourbaki2002lie} for a precise definition.

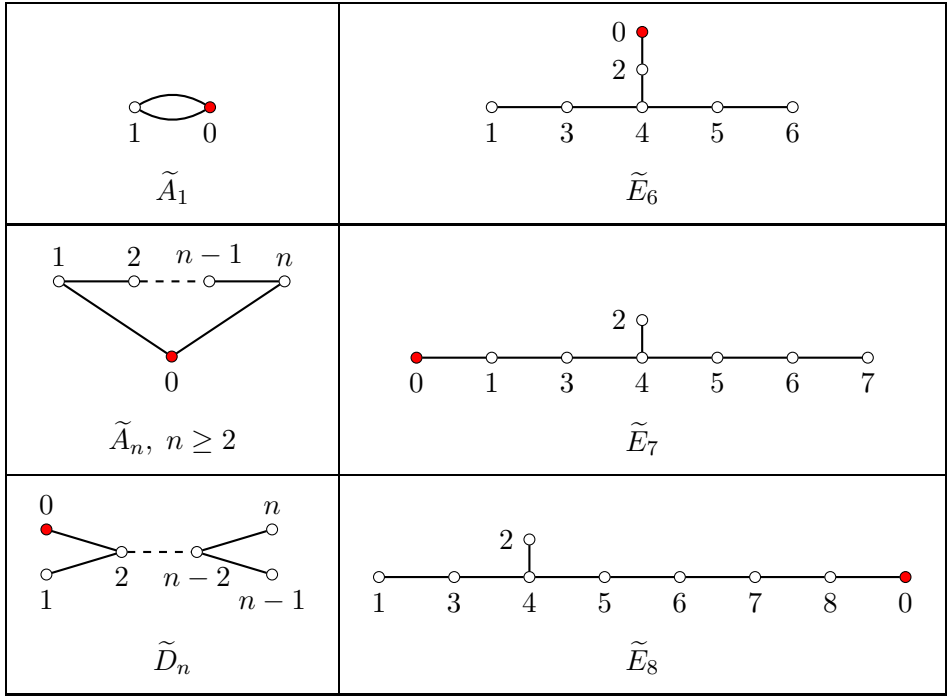
\begin{figure}
  \def\scl{0.4}
\begin{tabular}{|c|c|}\hline
\begin{tikzpicture}
\node[anchor=north] (A2) at (0,0) {\begin{tikzpicture}
\node[scale=\scl,draw,circle] (1) at (-1,0) {};
\node[scale=\scl,draw,circle,fill=red] (0) at (0,0) {};
\node[anchor=north] at (1.south) {$1$};
\node[anchor=north] at (0.south) {$0$};
\draw[thick] (1) to [bend left] (0);
\draw[thick] (1) to [bend right] (0);
\end{tikzpicture}};
\node[anchor=north] (A2l) at (A2.south) {$\widetilde{A}_1$};
\end{tikzpicture} &
\begin{tikzpicture}
\node[anchor=north] (E6) at (6,0) {\begin{tikzpicture}
\node[scale=\scl,draw,circle] (1) at (-1,0) {};
\node[scale=\scl,draw,circle] (2) at (1,0.5) {};
\node[scale=\scl,draw,circle] (3) at (0,0) {};
\node[scale=\scl,draw,circle] (4) at (1,0) {};
\node[scale=\scl,draw,circle] (5) at (2,0) {};
\node[scale=\scl,draw,circle] (6) at (3,0) {};
\node[scale=\scl,draw,circle,fill=red] (0) at (1,1) {};
\node[anchor=north] at (1.south) {$1$};
\node[anchor=east] at (2.west) {$2$};
\node[anchor=north] at (3.south) {$3$};
\node[anchor=north] at (4.south) {$4$};
\node[anchor=north] at (5.south) {$5$};
\node[anchor=north] at (6.south) {$6$};
\node[anchor=east] at (0.west) {$0$};
\draw[thick] (1)--(3)--(4)--(5)--(6);
\draw[thick] (0)--(2)--(4);
\end{tikzpicture}};
\node[anchor=north] (E6l) at (E6.south) {$\widetilde{E}_6$};
\end{tikzpicture}\\\hline
\begin{tikzpicture}
\node[anchor=north] (A) at (0,0) {\begin{tikzpicture}
\node[scale=\scl,draw,circle] (1) at (-1,0) {};
\node[scale=\scl,draw,circle] (2) at (0,0) {};
\node[scale=\scl,draw,circle] (3) at (1,0) {};
\node[scale=\scl,draw,circle] (4) at (2,0) {};
\node[scale=\scl,draw,circle,fill=red] (0) at (0.5,-1) {};
\node[anchor=south] at (1.north) {$1$};
\node[anchor=south] at (2.north) {$2$};
\node[anchor=south] at (3.north) {${n-1}$};
\node[anchor=south] at (4.north) {$n$};
\node[anchor=north] at (0.south) {$0$};
\draw[thick] (1)--(2);
\draw[thick,dashed] (2)--(3);
\draw[thick] (3)--(4);
\draw[thick] (1)--(0)--(4);
\end{tikzpicture}};
\node[anchor=north] (Al) at (A.south) {$\widetilde{A}_n, \; n\geq 2$};
\end{tikzpicture}&
\begin{tikzpicture}
\node[anchor=north] (E7) at (E6l.south) {\begin{tikzpicture}
\node[scale=\scl,draw,circle] (1) at (-1,0) {};
\node[scale=\scl,draw,circle] (2) at (1,0.5) {};
\node[scale=\scl,draw,circle] (3) at (0,0) {};
\node[scale=\scl,draw,circle] (4) at (1,0) {};
\node[scale=\scl,draw,circle] (5) at (2,0) {};
\node[scale=\scl,draw,circle] (6) at (3,0) {};
\node[scale=\scl,draw,circle] (7) at (4,0) {};
\node[scale=\scl,draw,circle,fill=red] (0) at (-2,0) {};
\node[anchor=north] at (1.south) {$1$};
\node[anchor=east] at (2.west) {$2$};
\node[anchor=north] at (3.south) {$3$};
\node[anchor=north] at (4.south) {$4$};
\node[anchor=north] at (5.south) {$5$};
\node[anchor=north] at (6.south) {$6$};
\node[anchor=north] at (7.south) {$7$};
\node[anchor=north] at (0.south) {$0$};
\draw[thick] (0)--(1)--(3)--(4)--(5)--(6)--(7);
\draw[thick] (2)--(4);
\end{tikzpicture}};
\node[anchor=north] (E7l) at (E7.south) {$\widetilde{E}_7$};
\end{tikzpicture}\\\hline
\begin{tikzpicture}
\node[anchor=north] (D) at (0,0) {\begin{tikzpicture}
\node[scale=\scl,draw,circle] (1) at (-1,-0.3) {};
\node[scale=\scl,draw,circle] (2) at (0,0) {};
\node[scale=\scl,draw,circle] (3) at (1,0) {};
\node[scale=\scl,draw,circle] (4) at (2,-0.3) {};
\node[scale=\scl,draw,circle] (5) at (2,0.3) {};
\node[scale=\scl,draw,circle,fill=red] (0) at (-1,0.3) {};
\node[anchor=north] at (1.south) {$1$};
\node[anchor=north] at (2.south) {$2$};
\node[anchor=north] at (3.south) {${n-2}$};
\node[anchor=north] at (4.south) {${n-1}$};
\node[anchor=south] at (5.north) {$n$};
\node[anchor=south] at (0.north) {$0$};
\draw[thick] (1)--(2);
\draw[thick] (0)--(2);
\draw[thick,dashed] (2)--(3);
\draw[thick] (3)--(4);
\draw[thick] (3)--(5);
\end{tikzpicture}};
\node[anchor=north] (Dl) at (D.south) {$\widetilde{D}_n$};
\end{tikzpicture} &
\begin{tikzpicture}
\node[anchor=north] (E8) at (E7l.south) {\begin{tikzpicture}
\node[scale=\scl,draw,circle] (1) at (-1,0) {};
\node[scale=\scl,draw,circle] (2) at (1,0.5) {};
\node[scale=\scl,draw,circle] (3) at (0,0) {};
\node[scale=\scl,draw,circle] (4) at (1,0) {};
\node[scale=\scl,draw,circle] (5) at (2,0) {};
\node[scale=\scl,draw,circle] (6) at (3,0) {};
\node[scale=\scl,draw,circle] (7) at (4,0) {};
\node[scale=\scl,draw,circle] (8) at (5,0) {};
\node[scale=\scl,draw,circle,fill=red] (0) at (6,0) {};
\node[anchor=north] at (1.south) {$1$};
\node[anchor=east] at (2.west) {$2$};
\node[anchor=north] at (3.south) {$3$};
\node[anchor=north] at (4.south) {$4$};
\node[anchor=north] at (5.south) {$5$};
\node[anchor=north] at (6.south) {$6$};
\node[anchor=north] at (7.south) {$7$};
\node[anchor=north] at (8.south) {$8$};
\node[anchor=north] at (0.south) {$0$};
\draw[thick] (1)--(3)--(4)--(5)--(6)--(7)--(8)--(0);
\draw[thick] (2)--(4);
\end{tikzpicture}};
\node[anchor=north] (E8l) at (E8.south) {$\widetilde{E}_8$};
\end{tikzpicture}\\\hline
\end{tabular}
\caption{The affine Dynkin diagrams. The ``affine node'' $0$ is filled in red.} \label{fig:affinedds}
\end{figure}

We also need the following lemma relating affine Dynkin diagrams to highest roots. 

\begin{lemma}[{See~\cite[VI, \S3]{bourbaki2002lie}}] \label{lem:affinehighest}
 If $\Phi$ is simply laced and $\DD$ is its Dynkin diagram, then we have~$\theta= \sum_{i=1}^{n}c_i\omega_i$, where $c_i$ is the number of edges between $i$ and $0$ in~$\DDaff$.
\end{lemma}

\begin{definition}\label{dfn:UCF}
Let $\DD$ be a simply laced Dynkin diagram with vertex set $[n]$. Let $\fund:[n]\to \mathbb{Z}_{\geq0}$ be an assignment of nonnegative integers to the vertices of $\DD$. An \emph{unlabeled central-firing move} (a \emph{UCF move} for short) is an application of the following sequence of steps to $\fund$:
  \begin{enumerate}[\normalfont(1)]
  \item choose a \emph{zero connected component} $\compFund$ of $\fund$, that is, a connected component of the induced subgraph of $\DD$ with vertex set $\{i\in [n]\colon \fund(i)=0\}$;
  \item\label{step:complete} complete $\compFund$ to an affine Dynkin diagram $\compFundaff$ with vertex set $R\cup\{\extraRoot\}$;
  \item\label{step:increase} for every edge $\{\extraRoot,i\}$ of $\compFundaff$, increase $\fund(i)$ by $1$;
    \item\label{step:decrease} for every vertex $j\notin\compFund$ that is adjacent to a vertex $i\in\compFund$, decrease $\fund(j)$ by $1$.
    \end{enumerate} 
    We denote the resulting assignment of integers by $\fund'$ and write $\fund\raFund\fund'$. We say that~$\fund'$ is obtained from $\fund$ \emph{via a UCF move along $\compFund$}.
  \end{definition}

  \begin{example}\label{example:unlabeled_E7}
  Let us illustrate this definition by an example for $\Phi$ of Type $E_7$. Consider an assignment $\fund$ shown in Figure~\ref{fig:E7_unlabeled_confluence} (top). It has two zero connected components: $\compFund_1$ of Type $D_5$ and $\compFund_2$ of Type $A_1$. Applying a UCF move to $\fund$ along $\compFund_1$ (resp., along $\compFund_2$) produces assignments $\fund_1'$ (resp., $\fund_2'$) shown in Figure~\ref{fig:E7_unlabeled_confluence} (middle-left), resp., (middle-right). Note that $\fund_1'$ has a zero connected component of Type $A_5$ that contains $\compFund_2$, and similarly, $\fund_2'$ has a zero connected component of Type $E_6$ that contains $\compFund_1$. Moreover, applying another UCF move along the corresponding zero connected component of $\fund_1'$ (resp., of $\fund_2'$) actually produces the same result $\fund''$ shown in Figure~\ref{fig:E7_unlabeled_confluence} (bottom).
  \end{example}
 
  \begin{figure}

    \def\scl{0.2}
    \def\tikzscl{0.8}
    \def\gammascl{1.2}
    \def\Eseven{
      \node[scale=\scl,draw,circle,fill=black] (A) at (-2,0) {};
      \node[scale=\scl,draw,circle,fill=black] (B) at (-1,0) {};
      \node[scale=\scl,draw,circle,fill=black] (C) at (0,0) {};
      \node[scale=\scl,draw,circle,fill=black]  (D) at (1,0) {};
      \node[scale=\scl,draw,circle,fill=black] (E) at (2,0) {};
      \node[scale=\scl,draw,circle,fill=black] (F) at (3,0) {};
      \node[scale=\scl,draw,circle,fill=black] (G) at (0,1) {};
      \draw (A)--(B)--(C)--(D)--(E)--(F);
      \draw (C)--(G);
    }
    \newcommand{\Esevenargs}[7]{
      \Eseven
      \node[anchor=north] (X) at (A.south) {$#1$};
      \node[anchor=north] (X) at (B.south) {$#2$};
      \node[anchor=north] (X) at (C.south) {$#3$};
      \node[anchor=north] (X) at (D.south) {$#4$};
      \node[anchor=north] (X) at (E.south) {$#5$};
      \node[anchor=north] (X) at (F.south) {$#6$};
      \node[anchor=west] (X) at (G.east) {$#7$};
    }
    \def\dy{-0.5}
    \def\dx{0.7}
    \def\intscale{0.7}
    \def\intscaletext{0.7}
    \def\circlerad{4pt}
    \begin{tikzpicture}[scale=0.6]
      \node (top) at (0,0) {
        \begin{tikzpicture} 
          \node[scale=\gammascl,anchor=east] (GM) at (-2.5,0) {$\fund=$};
          \Esevenargs 0 0 0 0 1 0 0
        \end{tikzpicture}
      };
      \node[anchor=north east] (topleft) at ($(top.-110)+(-\dx,\dy)$) {
        \begin{tikzpicture}[scale=\intscale]
          \Eseven
          \draw[line width=1.5pt,draw=blue] (A)--(B)--(C)--(D);
          \draw[line width=1.5pt,draw=blue] (C)--(G);
          \coordinate (Z) at (-1,-1);
          \draw[line width=1.5pt,draw=red] (B)--(Z);
          
          \draw[fill=blue,draw=black,thick] (A) circle (\circlerad);
          \draw[fill=blue,draw=black,thick] (B) circle (\circlerad);
          \draw[fill=blue,draw=black,thick] (C) circle (\circlerad);
          \draw[fill=blue,draw=black,thick] (D) circle (\circlerad);
          \draw[fill=blue,draw=black,thick] (G) circle (\circlerad);
          \draw[fill=red,draw=black,thick] (Z) circle (\circlerad);
          \node[anchor=south,red,scale=\intscaletext] (RR) at (B.north) {$+1$}; 
          \node[anchor=south,blue,scale=\intscaletext] (BB) at (E.north) {$-1$};          
        \end{tikzpicture}
      };
      \node[anchor=north west] (topright) at ($(top.-110)+(\dx,\dy)$) {
        \begin{tikzpicture}[scale=\intscale]
          \Eseven
          \coordinate (Z) at (3,1);
          \coordinate (ZZ) at (2,-1);
          \draw[line width=1.5pt,draw=red] (F) to[bend right=20] (Z);
          \draw[line width=1.5pt,draw=red] (F)  to[bend left=20] (Z);
          \draw[fill=blue,draw=black,thick] (F) circle (\circlerad);
          \draw[fill=red,draw=black,thick] (Z) circle (\circlerad);
          \draw[draw=white,thick] (ZZ) circle (\circlerad);
          \node[anchor=north,red,scale=\intscaletext] (RR) at (F.south) {$+2$}; 
          \node[anchor=south,blue,scale=\intscaletext] (BB) at (E.north) {$-1$};  
          
        \end{tikzpicture}
      };
      \node[anchor=north] (bottomleft) at  ($(topleft.south)+(-3*\dx,\dy)$) {
        \begin{tikzpicture}
          \node[scale=\gammascl,anchor=east] (GM) at (-2.5,0) {$\fund'_1=$};
          \Esevenargs 0 1 0 0 0 0 0
        \end{tikzpicture}
      };
      \node[anchor=north] (bottomright) at ($(topright.south)+(3*\dx,\dy)$) {
        \begin{tikzpicture}
          \node[scale=\gammascl,anchor=west] (GM) at (3.5,0) {$=\fund'_2$};
          \Esevenargs 0 0 0 0 0 2 0
        \end{tikzpicture}
      };
      \node[anchor=north] (bottomleftdown) at  ($(bottomleft.south)+(3*\dx,\dy)$) {
        \begin{tikzpicture}[scale=\intscale]
          \Eseven
          \draw[line width=1.5pt,draw=blue] (C)--(D)--(E)--(F);
          \draw[line width=1.5pt,draw=blue] (C)--(G);
          \coordinate (Z) at (2,1);
          \coordinate (ZZ) at (2,2);
          \draw[line width=1.5pt,draw=red] (G)--(Z);
          \draw[line width=1.5pt,draw=red] (F)--(Z);
          
          \draw[fill=blue,draw=black,thick] (C) circle (\circlerad);
          \draw[fill=blue,draw=black,thick] (D) circle (\circlerad);
          \draw[fill=blue,draw=black,thick] (E) circle (\circlerad);
          \draw[fill=blue,draw=black,thick] (F) circle (\circlerad);
          \draw[fill=blue,draw=black,thick] (G) circle (\circlerad);
          \draw[fill=red,draw=black,thick] (Z) circle (\circlerad);
          \draw[draw=white,thick] (ZZ) circle (\circlerad);
          \node[anchor=south,red,scale=\intscaletext] (RR) at (G.north) {$+1$}; 
          \node[anchor=north,red,scale=\intscaletext] (RR) at (F.south) {$+1$}; 
          \node[anchor=south,blue,scale=\intscaletext] (BB) at (B.north) {$-1$};
        \end{tikzpicture}
      };
      \node[anchor=north] (bottomrightdown) at ($(bottomright.south)+(-3*\dx,\dy)$) {
        \begin{tikzpicture}[scale=\intscale]
          \Eseven
          \draw[line width=1.5pt,draw=blue] (A)--(B)--(C)--(D)--(E);
          \draw[line width=1.5pt,draw=blue] (C)--(G);
          \coordinate (Z) at (0,2);
          \draw[line width=1.5pt,draw=red] (G)--(Z);
          
          \draw[fill=blue,draw=black,thick] (A) circle (\circlerad);
          \draw[fill=blue,draw=black,thick] (B) circle (\circlerad);
          \draw[fill=blue,draw=black,thick] (C) circle (\circlerad);
          \draw[fill=blue,draw=black,thick] (D) circle (\circlerad);
          \draw[fill=blue,draw=black,thick] (E) circle (\circlerad);
          \draw[fill=blue,draw=black,thick] (G) circle (\circlerad);
          \draw[fill=red,draw=black,thick] (Z) circle (\circlerad);
          \node[anchor=west,red,scale=\intscaletext] (RR) at (G.east) {$+1$}; 
          \node[anchor=south,blue,scale=\intscaletext] (BB) at (F.north) {$-1$};
        \end{tikzpicture}
      };
      
      \path let \p1 = ($(bottomleftdown.south)+(0,\dy)$) in node[anchor=north] (down)  at (0,\y1) {
        \begin{tikzpicture}
          \node[scale=\gammascl,anchor=east] (GM) at (-2.5,0) {$\fund''=$};
          \Esevenargs 0 0 0 0 0 1 1
        \end{tikzpicture}
      };
      \draw[->,dashed] (top)--(topleft);
      \draw[->,dashed] (topleft)--(bottomleft);
      \draw[->,dashed] (bottomleft)--(bottomleftdown);
      \draw[->,dashed] (bottomleftdown)--(down);
      \draw[->,dashed] (top)--(topright);
      \draw[->,dashed] (topright)--(bottomright);
      \draw[->,dashed] (bottomright)--(bottomrightdown);
      \draw[->,dashed] (bottomrightdown)--(down);
    \end{tikzpicture}
    \caption{Applying UCF moves to the Dynkin diagram of $E_7$ (see Example~\ref{example:unlabeled_E7}). For each move, the component $\compFund$ is shown in blue, the extra node $\extraRoot$ of $\compFundaff$ is shown in red, changes from step~\eqref{step:increase} are shown in red,  and changes from step~\eqref{step:decrease} are shown in blue.} \label{fig:E7_unlabeled_confluence} 
  \end{figure}
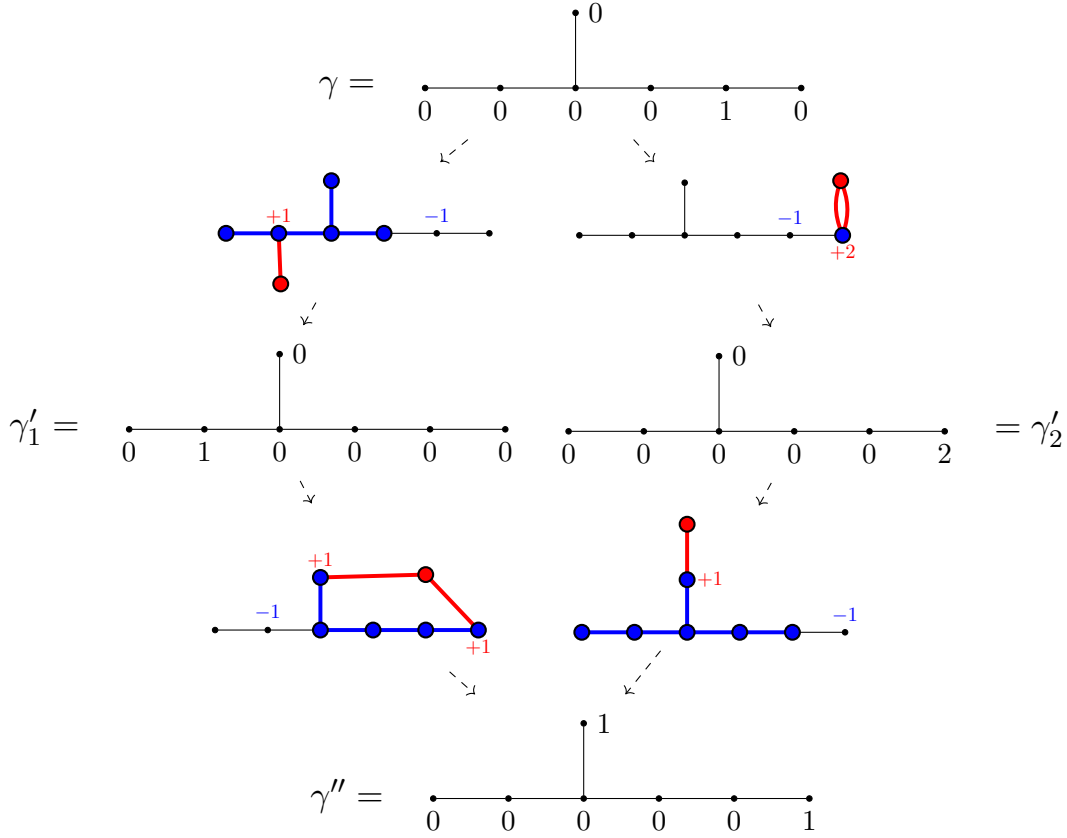

\begin{example}
All states of the classical chip-firing process starting with four chips at the origin are shown on the left of Figure~\ref{fig:side_by_side}; meanwhile, all states of the unlabeled central-firing process starting from $0$ in Type $A_3$ are shown on the right of Figure~\ref{fig:side_by_side}.
\end{example}

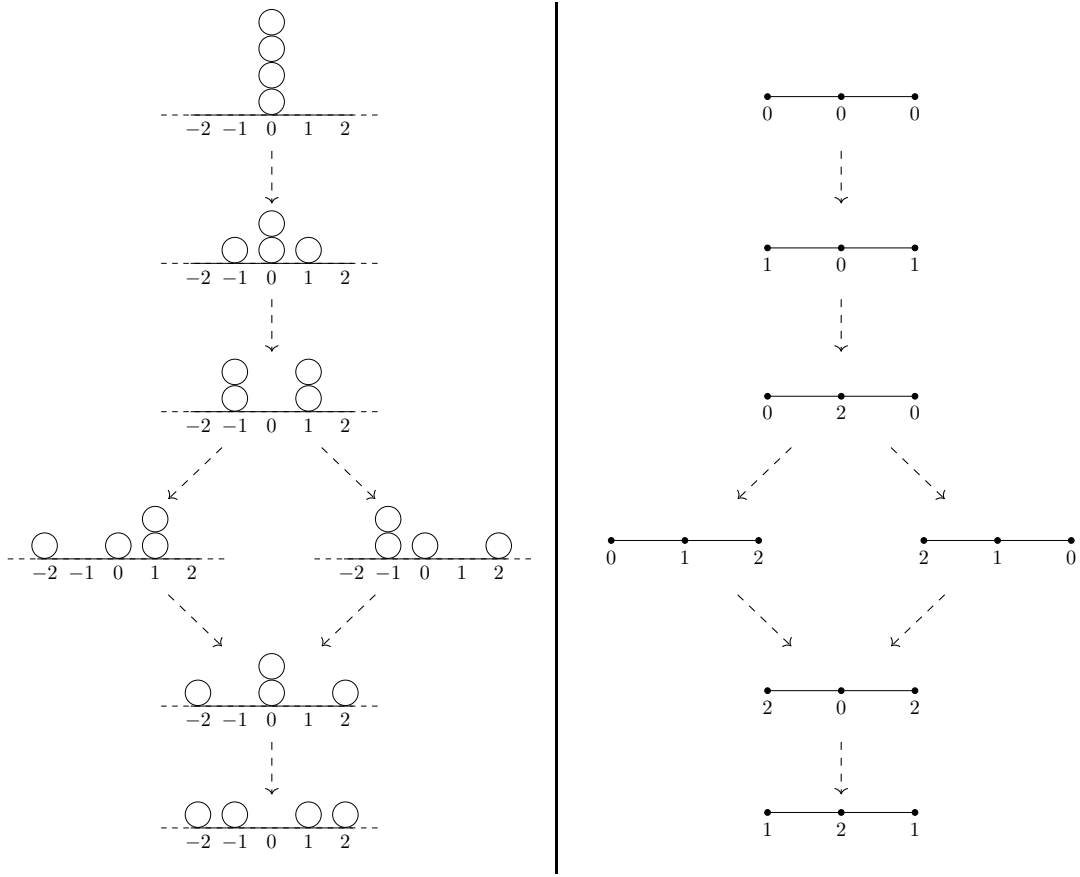
\begin{figure}
  \def\nodesclbig{0.7}
  \def\shift{0.7cm}
  \def\horshift{0.2cm}
  \def\sclAthree{0.3}
  \def\nodescaleAthree{1}
  \def\scaleAthree{2}

\newcommand{\AthreeArgs}[3]{
      \node[scale=\sclAthree,draw,circle,fill=black] (A) at (-\scaleAthree,0.5) {};
    \node[scale=\sclAthree,draw,circle,fill=black] (B) at (0,0.5) {};
    \node[scale=\sclAthree,draw,circle,fill=black] (C) at (\scaleAthree,0.5) {};
    \draw (A)--(B)--(C);
    \node[anchor=north,scale=\nodescaleAthree] (AA) at (A.south) {$#1$};
    \node[anchor=north,scale=\nodescaleAthree] (BB) at (B.south) {$#2$};
    \node[anchor=north,scale=\nodescaleAthree] (CC) at (C.south) {$#3$};
      }
  
  \begin{tabular}{c|c}
\begin{tikzpicture}
  \node[scale=\nodesclbig] (top1) at (0,0) {
    \begin{tikzpicture}[scale=\chiptikzscl,block/.style={anchor=center,draw,circle, minimum width={width("11")+12pt},
font=\small,scale=\chipscl}]
    \foreach \x in {-2,-1,...,2} {%
      \node[anchor=north] (A\x) at (\x,0) {$\x$};
    }
    \draw (-2.2,0) -- (2.2,0);
    \draw[dashed] (-3,0) -- (3,0);
    \foreach[count=\i] \a/\b in {0/1,0/2,0/3,0/4} {%
      \node[block] at (\a,{\b*\chipscl*\chipcoef-0.5*\chipscl*\chipcoef}) { };
      }
\end{tikzpicture}
};

\node[scale=\nodesclbig,below=\shift of top1] (top2) {
    \begin{tikzpicture}[scale=\chiptikzscl,block/.style={anchor=center,draw,circle, minimum width={width("11")+12pt},
font=\small,scale=\chipscl}]
    \foreach \x in {-2,-1,...,2} {%
      \node[anchor=north] (A\x) at (\x,0) {$\x$};
    }
    \draw (-2.2,0) -- (2.2,0);
    \draw[dashed] (-3,0) -- (3,0);
    \foreach[count=\i] \a/\b in {1/1,0/1,0/2,-1/1} {%
      \node[block] at (\a,{\b*\chipscl*\chipcoef-0.5*\chipscl*\chipcoef}) { };
      }
\end{tikzpicture}
  };
\node[scale=\nodesclbig,below=\shift of top2] (top3) {
    \begin{tikzpicture}[scale=\chiptikzscl,block/.style={anchor=center,draw,circle, minimum width={width("11")+12pt},
font=\small,scale=\chipscl}]
    \foreach \x in {-2,-1,...,2} {%
      \node[anchor=north] (A\x) at (\x,0) {$\x$};
    }
    \draw (-2.2,0) -- (2.2,0);
    \draw[dashed] (-3,0) -- (3,0);
    \foreach[count=\i] \a/\b in {1/1,1/2,-1/1,-1/2} {%
      \node[block] at (\a,{\b*\chipscl*\chipcoef-0.5*\chipscl*\chipcoef}) { };
      }
\end{tikzpicture}
  };
\node[scale=\nodesclbig,below=\shift of top3] (top4) {
    \begin{tikzpicture}[scale=\chiptikzscl,block/.style={anchor=center,draw,circle, minimum width={width("11")+12pt},
font=\small,scale=\chipscl}]
    \foreach[count=\i] \a/\b in {0/0,0/1,0/2} {%
      \node[block,white] at (\a,{\b*\chipscl*\chipcoef-0.5*\chipscl*\chipcoef}) { };
      }
\end{tikzpicture}
  };
\node[scale=\nodesclbig,left=\horshift of top4] (left) {
    \begin{tikzpicture}[scale=\chiptikzscl,block/.style={anchor=center,draw,circle, minimum width={width("11")+12pt},
font=\small,scale=\chipscl}]
    \foreach \x in {-2,-1,...,2} {%
      \node[anchor=north] (A\x) at (\x,0) {$\x$};
    }
    \draw (-2.2,0) -- (2.2,0);
    \draw[dashed] (-3,0) -- (3,0);
    \foreach[count=\i] \a/\b in {1/1,1/2,0/1,-2/1} {%
      \node[block] at (\a,{\b*\chipscl*\chipcoef-0.5*\chipscl*\chipcoef}) { };
      }
\end{tikzpicture}
  };
\node[scale=\nodesclbig,right=\horshift of top4] (right) {
    \begin{tikzpicture}[scale=\chiptikzscl,block/.style={anchor=center,draw,circle, minimum width={width("11")+12pt},
font=\small,scale=\chipscl}]
    \foreach \x in {-2,-1,...,2} {%
      \node[anchor=north] (A\x) at (\x,0) {$\x$};
    }
    \draw (-2.2,0) -- (2.2,0);
    \draw[dashed] (-3,0) -- (3,0);
    \foreach[count=\i] \a/\b in {2/1,0/1,-1/1,-1/2} {%
      \node[block] at (\a,{\b*\chipscl*\chipcoef-0.5*\chipscl*\chipcoef}) { };
      }
\end{tikzpicture}
  };
\node[scale=\nodesclbig,below=\shift of top4] (bottom1) {
    \begin{tikzpicture}[scale=\chiptikzscl,block/.style={anchor=center,draw,circle, minimum width={width("11")+12pt},
font=\small,scale=\chipscl}]
    \foreach \x in {-2,-1,...,2} {%
      \node[anchor=north] (A\x) at (\x,0) {$\x$};
    }
    \draw (-2.2,0) -- (2.2,0);
    \draw[dashed] (-3,0) -- (3,0);
    \foreach[count=\i] \a/\b in {2/1,0/1,0/2,-2/1} {%
      \node[block] at (\a,{\b*\chipscl*\chipcoef-0.5*\chipscl*\chipcoef}) { };
      }
\end{tikzpicture}
  };
\node[scale=\nodesclbig,below=\shift of bottom1] (bottom2) {
    \begin{tikzpicture}[scale=\chiptikzscl,block/.style={anchor=center,draw,circle, minimum width={width("11")+12pt},
font=\small,scale=\chipscl}]
    \foreach \x in {-2,-1,...,2} {%
      \node[anchor=north] (A\x) at (\x,0) {$\x$};
    }
    \draw (-2.2,0) -- (2.2,0);
    \draw[dashed] (-3,0) -- (3,0);
    \foreach[count=\i] \a/\b in {2/1,1/1,-1/1,-2/1} {%
      \node[block] at (\a,{\b*\chipscl*\chipcoef-0.5*\chipscl*\chipcoef}) { };
      }
\end{tikzpicture}
  };

  \draw[->,dashed] (top1)--(top2);
  \draw[->,dashed] (top2)--(top3);
  \draw[->,dashed] (top3)--(left);
  \draw[->,dashed] (top3)--(right);
  \draw[->,dashed] (left)--(bottom1);
  \draw[->,dashed] (right)--(bottom1);
  \draw[->,dashed] (bottom1)--(bottom2);

\end{tikzpicture}
    &

\begin{tikzpicture}
  \node[scale=\nodesclbig] (top1) at (0,0) {
    \begin{tikzpicture}[scale=\chiptikzscl,block/.style={anchor=center,draw,circle, minimum width={width("11")+12pt},
font=\small,scale=\chipscl}]
    \foreach \x in {-2,-1,...,2} {%
      \node[white,anchor=north] (A\x) at (\x,0) {$\x$};
    }
    \draw[white] (-2.2,0) -- (2.2,0);
    \draw[dashed,white] (-3,0) -- (3,0);
    \foreach[count=\i] \a/\b in {0/1,0/2,0/3,0/4} {%
      \node[white,block] at (\a,{\b*\chipscl*\chipcoef-0.5*\chipscl*\chipcoef}) { };
    }
    \AthreeArgs 0 0 0
     
\end{tikzpicture}
};

\node[scale=\nodesclbig,below=\shift of top1] (top2) {
    \begin{tikzpicture}[scale=\chiptikzscl,block/.style={anchor=center,draw,circle, minimum width={width("11")+12pt},
font=\small,scale=\chipscl}]
    \foreach \x in {-2,-1,...,2} {%
      \node[white,anchor=north] (A\x) at (\x,0) {$\x$};
    }
    \draw[white] (-2.2,0) -- (2.2,0);
    \draw[white,dashed] (-3,0) -- (3,0);
    \foreach[count=\i] \a/\b in {1/1,0/1,0/2,-1/1} {%
      \node[white,block] at (\a,{\b*\chipscl*\chipcoef-0.5*\chipscl*\chipcoef}) { };
    }

    \AthreeArgs 1 0 1
\end{tikzpicture}
  };
\node[scale=\nodesclbig,below=\shift of top2] (top3) {
    \begin{tikzpicture}[scale=\chiptikzscl,block/.style={anchor=center,draw,circle, minimum width={width("11")+12pt},
font=\small,scale=\chipscl}]
    \foreach \x in {-2,-1,...,2} {%
      \node[white,anchor=north] (A\x) at (\x,0) {$\x$};
    }
    \draw[white] (-2.2,0) -- (2.2,0);
    \draw[white,dashed] (-3,0) -- (3,0);
    \foreach[count=\i] \a/\b in {1/1,1/2,-1/1,-1/2} {%
      \node[white,block] at (\a,{\b*\chipscl*\chipcoef-0.5*\chipscl*\chipcoef}) { };
      }
    \AthreeArgs 0 2 0
\end{tikzpicture}
  };
\node[scale=\nodesclbig,below=\shift of top3] (top4) {
    \begin{tikzpicture}[scale=\chiptikzscl,block/.style={anchor=center,draw,circle, minimum width={width("11")+12pt},
font=\small,scale=\chipscl}]
    \foreach[count=\i] \a/\b in {0/0,0/1,0/2} {%
      \node[block,white] at (\a,{\b*\chipscl*\chipcoef-0.5*\chipscl*\chipcoef}) { };
      }
\end{tikzpicture}
  };
\node[scale=\nodesclbig,left=\horshift of top4] (left) {
    \begin{tikzpicture}[scale=\chiptikzscl,block/.style={anchor=center,draw,circle, minimum width={width("11")+12pt},
font=\small,scale=\chipscl}]
    \foreach \x in {-2,-1,...,2} {%
      \node[white,anchor=north] (A\x) at (\x,0) {$\x$};
    }
    \draw[white] (-2.2,0) -- (2.2,0);
    \draw[white,dashed] (-3,0) -- (3,0);
    \foreach[count=\i] \a/\b in {1/1,1/2,0/1,-2/1} {%
      \node[white,block] at (\a,{\b*\chipscl*\chipcoef-0.5*\chipscl*\chipcoef}) { };
    }
    
    \AthreeArgs 0 1 2
\end{tikzpicture}
  };
\node[scale=\nodesclbig,right=\horshift of top4] (right) {
    \begin{tikzpicture}[scale=\chiptikzscl,block/.style={anchor=center,draw,circle, minimum width={width("11")+12pt},
font=\small,scale=\chipscl}]
    \foreach \x in {-2,-1,...,2} {%
      \node[white,anchor=north] (A\x) at (\x,0) {$\x$};
    }
    \draw[white] (-2.2,0) -- (2.2,0);
    \draw[white,dashed] (-3,0) -- (3,0);
    \foreach[count=\i] \a/\b in {2/1,0/1,-1/1,-1/2} {%
      \node[white,block] at (\a,{\b*\chipscl*\chipcoef-0.5*\chipscl*\chipcoef}) { };
    }
    
    \AthreeArgs 2 1 0
\end{tikzpicture}
  };
\node[scale=\nodesclbig,below=\shift of top4] (bottom1) {
    \begin{tikzpicture}[scale=\chiptikzscl,block/.style={anchor=center,draw,circle, minimum width={width("11")+12pt},
font=\small,scale=\chipscl}]
    \foreach \x in {-2,-1,...,2} {%
      \node[white,anchor=north] (A\x) at (\x,0) {$\x$};
    }
    \draw[white] (-2.2,0) -- (2.2,0);
    \draw[white,dashed] (-3,0) -- (3,0);
    \foreach[count=\i] \a/\b in {2/1,0/1,0/2,-2/1} {%
      \node[white,block] at (\a,{\b*\chipscl*\chipcoef-0.5*\chipscl*\chipcoef}) { };
    }
    
    \AthreeArgs 2 0 2
\end{tikzpicture}
  };
\node[scale=\nodesclbig,below=\shift of bottom1] (bottom2) {
    \begin{tikzpicture}[scale=\chiptikzscl,block/.style={anchor=center,draw,circle, minimum width={width("11")+12pt},
font=\small,scale=\chipscl}]
    \foreach \x in {-2,-1,...,2} {%
      \node[white,anchor=north] (A\x) at (\x,0) {$\x$};
    }
    \draw[white] (-2.2,0) -- (2.2,0);
    \draw[white,dashed] (-3,0) -- (3,0);
    \foreach[count=\i] \a/\b in {2/1,1/1,-1/1,-2/1} {%
      \node[white,block] at (\a,{\b*\chipscl*\chipcoef-0.5*\chipscl*\chipcoef}) { };
    }
    
    \AthreeArgs 1 2 1
\end{tikzpicture}
  };

  \draw[->,dashed] (top1)--(top2);
  \draw[->,dashed] (top2)--(top3);
  \draw[->,dashed] (top3)--(left);
  \draw[->,dashed] (top3)--(right);
  \draw[->,dashed] (left)--(bottom1);
  \draw[->,dashed] (right)--(bottom1);
  \draw[->,dashed] (bottom1)--(bottom2);

\end{tikzpicture}
  \end{tabular}

  \caption{\label{fig:side_by_side} Applying classical chip-firing moves to four chips at the origin (left). Applying UCF moves to $0\in P$ for $\Phi$ of Type $A_3$ (right).}
\end{figure}

  It turns out that the UCF moves always ``commute,'' and define a binary relation that coincides with $\raPP$:

\def\fundweight{\omega}
\begin{thm}\label{thm:abelian}
  Let $\DD$ be a simply laced Dynkin diagram corresponding to the root system $\Phi$. For each assignment $\fund:[n]\to\mathbb{Z}_{\geq0}$ we define the corresponding dominant weight $\lambda(\fund):=\sum_{i=1}^{n}\fund(i)\fundweight_i$. Then:
  \begin{enumerate}[\normalfont(i)]
  \item\label{item:rapp_rafund} An assignment $\fund'$ is obtained from $\fund$ by a UCF move (i.e. $\fund\raFund\fund'$) if and only if we have $\W\lambda(\fund)\raPP\W\lambda(\fund')$.
    \item\label{item:ucf_commute} UCF moves always ``commute.'' More precisely, let $\compFund_1$ and $\compFund_2$ be two zero connected components of $\fund$, and let $\fund_1'$ (resp., $\fund_2'$) be the assignment obtained from~$\fund$ by a UCF move along $\compFund_1$ (resp., along $\compFund_2$). Then $\fund_1'$ has a zero connected component $\compFund_2'\supseteq\compFund_2$, $\fund_2'$ has a zero connected component $\compFund_1'\supseteq\compFund_1$, and applying a UCF move to $\fund_1'$ along $\compFund_2'$ produces the same result as applying a UCF move to~$\fund_2'$ along $\compFund_1'$. 
  \end{enumerate}
\end{thm}
\begin{proof}
  We start with~\eqref{item:rapp_rafund}. Recall from the proof of Proposition~\ref{prop:simply_laced_dominant} that if $\lambda$ is dominant then for any root $\beta\in\Phi$ such that $\<\lambda,\beta^\vee\>=0$, there exists a root which we denote~$\theta'$ such that $\lambda+\theta'$ is dominant and $\W(\lambda+\beta)=\W(\lambda+\theta')$. Moreover, it is easy to see again from the proof of Proposition~\ref{prop:simply_laced_dominant} that such a root $\theta'$ is unique: it is the highest root of the irreducible sub-root system $\Phi'$ of $\Phi_{I^{0}_{\lambda}}$ containing $\beta$. Now let $\fund$ be such that~$\lambda=\sum_{i=1}^{n}\fund(i)\fundweight_i$. It is a simple, well-known fact that for every $\beta\in\Phi^+$ given by~$\beta=\sum_{i=1}^{n}b_i\alpha_i$, the graph $\DD[\beta]:=\DD[\{i\in [n]\colon b_i\neq 0\}]$ is connected. Thus we have that~$\<\lambda,\beta\>=0$ if and only if $\DD[\beta]$ is contained in a zero connected component $\compFund$ of $\fund$. It is then easy to see that the simple roots of $\Phi'$ are precisely $\{\alpha_i\colon i\in \compFund\}$. It remains to note that the highest root of $\Phi'$, written in the coordinates of the fundamental weights, is exactly given by steps~\eqref{step:increase} and~\eqref{step:decrease} of Definition~\ref{dfn:UCF}. In other words, we have
  \[\<\theta',\alpha^\vee_i\>=
    \begin{cases}
      -1, &\text{if $i\notin\compFund$ is connected to a vertex $j\in\compFund$;}\\
      d_{i,0},&\text{if $i\in\compFund$ and there are $d_{i,0}$ edges of $\compFundaff$ between $i$ and $0$;}\\ 
      0,&\text{otherwise.}
    \end{cases}\]
That $\<\theta',\alpha^\vee_i\>= d_{i,0}$ if $i\in\compFund$ follows from Lemma~\ref{lem:affinehighest} above. If~$i\notin\compFund$ is connected to a vertex $j\in\compFund$, then, writing $\theta'=\sum_{l=1}^{n}c_l\alpha_l$, we will have $c_j > 0$ since $\theta'$ is the highest weight of $\Phi'$; meanwhile, clearly $c_i=0$; hence, $\<\theta',\alpha^\vee_i\> < 0$; but since $\Phi$ is simply laced this means that $\<\theta',\alpha^\vee_i\> =-1$. That $\<\theta',\alpha^\vee_i\>= 0$ if $i\notin\compFund$ is not connected to any vertex in $\compFund$ is clear. This finishes the proof of~\eqref{item:rapp_rafund}. 

To show~\eqref{item:ucf_commute}, note that the moves $\W\lambda \raPP \W\mu_1$ and $\W\lambda \raPP\W\mu_2$ ``commute'' for any rank $2$ simply laced root system: in $A_2$, there is only one class of roots modulo the Weyl group, and in $A_1\oplus A_1$ there are two classes but the two possible moves do indeed ``commute.'' Thus part~\eqref{item:ucf_commute} follows from part~\eqref{item:rapp_rafund} as an immediate corollary.
\end{proof}

\begin{remark}
To extend this Dynkin diagram number game for unlabeled central-firing beyond the simply laced setting, there are two obstacles that need to be overcome. The first is that in general we may have both a highest root $\theta'$ and highest short root $\widehat{\theta}'$ for the parabolic sub-root system corresponding to a zero connected component of our weight, and adding $\theta'$ and $\widehat{\theta}'$ will lead to different weights. This is not such a serious obstacle: we can just allow these two different kinds of moves. The second, more serious, obstacle is that, as mentioned in Remark~\ref{rem:nonsimplylaced_central}, not every unlabeled central-firing move corresponds to a move that stays in the dominant chamber: thus, sometimes adding~$\theta'$ or~$\widehat{\theta}'$ will make some coordinates of our weight negative. To overcome this, we could reflect our weight back into the dominant chamber by playing what is called \emph{Mozes's number game} (see~\cite{mozes1990reflection} or~\cite{eriksson1996strong}) on our Dynkin diagram. But this second obstacle makes the description of the unlabeled chip-firing game much more convoluted than in the simply laced case.
\end{remark}

\section{Span of central-firing and connectedness}\label{sec:span-central-firing}

In this section, we try to reduce the study of confluence of central-firing to those weights from which the possible firing sequences ``span'' the whole vector space $V$.

\subsection{The firing span}

Recall that $\raPoint$ denotes central-firing of all of the roots of $\Phi$.

\begin{definition}
  Let $\l$ be a weight. We define the \emph{firing span} of $\Phi^+$ (resp., of $\Phi$) at $\l$ to be $\SpanPP(\l) := \Span\{\lambda-\mu\colon\lambda \raPPAst \mu\}$ (resp., $\SpanPoint(v) := \Span\{\lambda-\mu\colon\lambda \raPointAst \mu\}$).
\end{definition}

Here are some elementary properties of $\SpanPP$ and $\SpanPoint$.

\begin{prop}\label{prop:span_w_inv}
  Let $\lambda \in P$. Then:
  \begin{enumerate}[\normalfont(1)]
  \item\label{item:span_point} We have $\SpanPP(\lambda) = \SpanPoint(\lambda)$.
    \item\label{item:span_w} For any $w \in W$, we have $\SpanPP(w\lambda) = w\SpanPP(\lambda)$.
  \end{enumerate}
\end{prop}

\begin{proof}
We will use the following recurrence relation for $\SpanPP(\lambda)$ and $\SpanPoint(\lambda)$:
  \begin{equation*}\label{eq:recurrence_span}
  \SpanPP(\lambda)=\Span \hspace{-0.25cm} \bigcup_{\substack{\alpha\in\Phi^+\\ \<\lambda,\alpha^\vee\>=0}}  \hspace{-0.25cm} \{\alpha\}\cup\SpanPP(\lambda+\alpha); \quad \SpanPoint(\lambda)=\Span  \hspace{-0.25cm} \bigcup_{\substack{\alpha\in\Phi\\ \<\lambda,\alpha^\vee\>=0}}  \hspace{-0.25cm} \{\alpha\} \cup\SpanPoint(\lambda+\alpha).
  \end{equation*}
  We prove~\eqref{item:span_point} by induction on the length $\fl(\lambda)$ of the longest central-firing sequence starting at~$\lambda$ (it will be clear from our argument that $\fl(w\lambda)=\fl(\lambda)$ for all $w\in W$ and~$\lambda\in P$). The case $\fl(\lambda)=0$ is trivial so suppose that there exists $\alpha\in\Phi^+$ orthogonal to~$\lambda$. To show that~\eqref{item:span_point} holds for $\lambda$, it suffices to show that
 \[\Span(\{\alpha\}\cup\SpanPP(\lambda+\alpha))=\Span(\{\alpha\}\cup\SpanPoint(\lambda+\alpha))=\Span(\{-\alpha\}\cup\SpanPoint(\lambda-\alpha)).\]
  The first equality holds trivially by induction. Note that $\raPoint$ is clearly $W$-invariant; hence we have~$\SpanPoint(w\lambda) = w\SpanPoint(\lambda)$ for all~$\lambda \in P$ (thus~\eqref{item:span_w} in fact follows from~\eqref{item:span_point}). Using this, we get
  \[s_\alpha\SpanPoint(\lambda+\alpha)=\SpanPoint(\lambda-\alpha),\]
  while on the other hand,
  \[s_{\alpha}\SpanPoint(\lambda+\alpha)=\{v-\<v,\alpha^\vee\>\alpha\colon v\in\SpanPoint(\lambda+\alpha)\}
    \subseteq\Span (\{\alpha\}\cup\SpanPoint(\lambda+\alpha)).\]
This shows the second equality and thus finishes the inductive step for part~\eqref{item:span_point}; as we have already noted, part~\eqref{item:span_w} follows from part~\eqref{item:span_point}.
\end{proof}

The following definition was the main reason for introducing the firing span.

\begin{definition}
We say that a weight $\lambda \in P$ is \emph{connected} if $\SpanPP(\lambda)=V$.
\end{definition}

The term ``connected'' comes from the interpretation of this notion in terms of chips: in Type $A_{N-1}$, given a configuration of chips, consider a simple undirected graph $G$ with vertex set $[N]$ and edge set containing $\{i,j\}$ whenever the \chip i and \chip j can fire together in some labeled chip-firing sequence starting from this configuration. Then such a configuration of chips corresponds to a connected weight if and only if the above graph $G$ is connected.

If $\lambda \in P$ is not connected, then we can understand central-firing from $\lambda$ by projecting to $\SpanPP(\lambda)$. Hence, in some sense, we can reduce Question~\ref{question:centralconf} to the case where $\lambda$ is connected. Of course, in order to carry out this reduction, we need to be able to decide when $\lambda$ is connected and efficiently compute $\SpanPP(\lambda)$ when it is not connected. We do not know how to do this for general~$\Phi$. But the main result of this section is a classification of connected weights when $\Phi$ is of Type A (and this classification in fact leads to an efficient way to compute $\SpanPP(\lambda)$ for all $\lambda\in P$, as we describe in the following subsection).

By Proposition~\ref{prop:span_w_inv}, the set of connected weights is some $W$-invariant set. It would be nice if it were, say, the weights inside some permutohedron. We now work towards proving that, in Type A at least, this is the case.

\begin{prop}\label{prop:span_pi_2_rho}
Let $\lambda \in P$. If $\lambda \notin \Pi(2\rho)$, then $\lambda$ is not connected.
\end{prop}
\begin{proof}
This follows from~\cite[\permcc]{galashin2017rootfiring1} in a straightforward way.
\end{proof}

To continue the analysis of connected weights, we now restrict our attention to simply laced root systems. 

\begin{prop} \label{prop:span_simp}
Suppose that $\Phi$ is simply laced. Let $\lambda \in P$ be dominant. Then 
\[\SpanPP(\lambda)=\Span \{\alpha_i\colon \mu\raPP \mu+\alpha_i \textrm{ for some $i\in [n]$ and $\mu \in P$ such that } \lambda \raPPAst \mu\}.\]
\end{prop}
\begin{proof}
We prove this by induction on the length $\fl(\lambda)$ of the longest central-firing sequence starting at~$\lambda$. Recall that
\[ \SpanPP(\lambda)=\Span \hspace{-0.25cm} \bigcup_{\substack{\alpha\in\Phi^+\\ \<\lambda,\alpha^\vee\>=0}}  \hspace{-0.25cm} \{\alpha\}\cup\SpanPP(\lambda+\alpha).\]
Now, we know from Proposition~\ref{prop:simply_laced_dominant} that for any $\alpha \in \Phi^+$ with $\<\lambda,\alpha^\vee\>=0$, we have $\lambda \raPP \lambda + \theta'$ where $\theta' \in \Phi^+$, $\lambda+\theta'$ is dominant, and $w(\lambda + \theta') = \lambda+\alpha$ for some $w \in W_{I^{0}_{\lambda}}$. But we have $\SpanPP(w(\lambda+\theta')) = w\SpanPP(\lambda+\theta')$ thanks to Proposition~\ref{prop:span_w_inv}. And note that since $w \in W_{I^0_{\lambda}}$, in fact we have
\[w\SpanPP(\lambda+\theta') \subseteq \Span \hspace{-0.25cm} \bigcup_{\substack{\alpha\in\Phi^+\\ \<\lambda,\alpha^\vee\>=0}}  \hspace{-0.25cm} \{\alpha\} \cup  \SpanPP(\lambda+\theta').\]
Since $\lambda$ is dominant, $\Span\{\alpha \in \Phi^+\colon \<\lambda,\alpha^\vee\>=0\} = \Span\{\alpha_i \in \Delta\colon \<\lambda,\alpha_i^\vee\>=0\}$. Altogether this shows that
\[ \SpanPP(\lambda)=\Span \hspace{-0.25cm} \bigcup_{\substack{\alpha_i\in\Delta\\ \<\lambda,\alpha_i^\vee\>=0}}  \hspace{-0.25cm} \{\alpha_i\}\cup  \hspace{-0.4cm}\bigcup_{\substack{\theta' \in \Phi^+ \\ \<\lambda,(\theta')^\vee\>=0\\ \lambda+\theta' \textrm{ is dominant}}} \hspace{-0.4cm}\SpanPP(\lambda+\theta').\]
By induction the result holds for all these $\lambda+\theta'$, so we are done.
\end{proof}

For $\lambda \in P$, let us use $\Pi^{\circ}(\lambda)$ to denote the \emph{interior} of the permutohedron $\Pi(\lambda)$, i.e., $\Pi^{\circ}(\lambda) := \Pi(\lambda)-\partial\Pi(\lambda)$ where $\partial\Pi(\lambda)$ is the (topological) boundary of $\Pi(\lambda)$. Then let us also use~$\Pi^{\circ,Q}(\lambda) := \Pi^{\circ}(\lambda)\cap (Q+\lambda)$.

\begin{prop}\label{prop:span_pi_rho}
Suppose that $\Phi$ is simply laced. Let $\lambda \in P$. If $\lambda \in \Pi^{\circ,Q}(\rho+\omega)$ for some $\omega \in \Omega_m^{0}$, then $\lambda$ is connected.
\end{prop}
\begin{proof}
  By Proposition~\ref{prop:span_w_inv}, we may assume that $\lambda$ is dominant. By Propositions~\ref{prop:unlabeled_stabilization} and~\ref{prop:simply_laced_dominant}, there exists a sequence
  \[\lambda=\lambda_0 \raPP \lambda_1 \raPP \dots \raPP \lambda_t=\rho+\omega\]
  such that $\lambda_s$ is dominant for $0\leq s\leq t$. Now, the fact that $\lambda$ belongs to the interior of~$\Pi^{\circ,Q}(\lambda_t)$ means that for every $i\in [n]$ we have $\<\lambda_0,\omega_i\> < \<\lambda_t,\omega_i\>$ (because the fundamental weights $\omega_i$ are the normals to the facets of $\Pi(\lambda_t)$ containing the vertex~$\lambda_t$). This means $(\rho+\omega) - \lambda = \sum_{i=1}^{n}a_i\alpha_i$ where $a_i \geq 1$ for all $i \in [n]$. But because $\sum_{i=1}^{n}a_i\alpha_i \in \SpanPP(\lambda)$, by Proposition~\ref{prop:span_simp} we conclude that $\alpha_i \in \SpanPP(\lambda)$ for all $i \in [n]$, thus proving the claim.
  \end{proof}

Thus for simply laced root systems, Propositions~\ref{prop:span_pi_2_rho} and~\ref{prop:span_pi_rho} tell us that the weights outside $\Pi(2\rho)$ are not connected while the weights inside $\Pi^{\circ,Q}(\rho+\omega)$ are connected for any $\omega\in \Omega^0_m$. In Type A we can show that actually the latter are the only connected weights:

\begin{prop}\label{prop:span_A}
Suppose $\Phi=A_n$. Then a weight $\lambda \in P$ is connected if and only if~$\lambda \in \Pi^{\circ,Q}(\rho+\omega)$ for some $\omega\in \Omega^0_m$. 
\end{prop}
\begin{proof}
By definition, in order to have $\lambda \in \Pi^{\circ,Q}(\rho+\omega)$, we must have $(\rho+\omega- \lambda)\in Q$, and this holds for exactly one element $\omega\in  \Omega^0_m$. We thus fix $\omega \in \Omega^0_m$ to be such that $(\rho+\omega- \lambda)\in Q$. By Proposition~\ref{prop:span_pi_rho}, we only need to show that if $\lambda\notin\Pi^{\circ,Q}(\rho+\omega)$ then $\lambda$ is not connected. And by Proposition~\ref{prop:span_w_inv}, we only need to consider the case when $\lambda$ is dominant. Recall that for $\Phi=A_n$, the simple roots are numbered as in Figure~\ref{fig:dynkinclassification}. So suppose that the dominant weight $\lambda$ has $\lambda \notin\Pi^{\circ,Q}(\rho+\omega)$, which means that for some $i_0\in [n]$ we have $\<\lambda,\omega_{i_0}\>\geq \<\rho+\omega,\omega_{i_0}\>\geq \<\rho,\omega_{i_0}\>$. 
  
  For a dominant weight $\nu\in P$, define $f_\nu:[0,n+1]\to\R_{\ge0}$ by $f_\nu(0)=f_\nu(n+1)=0$ and $f_\nu(i):=\<\nu,\omega_i\>$ for $1\leq i\leq n$. Thus $\nu=\sum_{i\in[n]}f_\nu(i)\alpha_i$, and for each $i\in[n]$, we have $\<\nu,\alpha^\vee_i\>=2f_\nu(i)-f_\nu(i-1)-f_\nu(i+1)$. Let $j\in [1,n]$ be such that $f_\lambda(i)-f_\rho(i)\leq f_\lambda(j)-f_\rho(j)$ for all $i\in [1,n]$. In particular, we have $f_\lambda(j)-f_\rho(j)\geq 0$ since $f_\lambda(i_0)-f_\rho(i_0)\geq 0$. 

We claim that $\alpha_j\notin\SpanPP(\lambda)$. To see that, suppose that $\alpha\in\Phi^+$ is orthogonal to $\lambda$, and let us write $\alpha=\alpha_a+\alpha_{a+1}+\dots+\alpha_b$ for some $1\leq a\leq b\leq n$. Then $\<\lambda,\alpha^\vee\>=0$ means that $\<\lambda,\alpha_i^\vee\>=0$ for all~$a\leq i\leq b$. Equivalently, the numbers $f_\lambda(a-1), f_\lambda(a),\dots,f_\lambda(b),f_\lambda(b+1)$ form an arithmetic progression, so the restriction of $f_\lambda$ to the interval $[a-1,b+1]$ is a linear function. On the other hand, we have $\<\rho,\alpha_i^\vee\>>0$ for all $i\in [n]$, so $f_\rho$ is a strictly concave function on $[0,n+1]$. Thus the function $f_\lambda-f_\rho$ is a strictly convex function on $[a-1,b+1]$ which therefore attains its maximum on one of the endpoints of this segment. More precisely, for $a\leq i\leq b$, we have $f_\lambda(i)-f_\rho(i)< f_\lambda(a-1)-f_\rho(a-1)$ or $f_\lambda(i)-f_\rho(i)< f_\lambda(b+1)-f_\rho(b+1)$. This shows that $j\notin [a,b]$ since $f_\lambda(j)-f_\rho(j)$ is the maximal value of $f_\lambda-f_\rho$. 

  Now, when we fire $\alpha$ from $\lambda$, we get
  \[f_{\lambda+\alpha}(i)=
    \begin{cases}
      f_\lambda(i)+1, &\text{if $a\leq i\leq b$;}\\
      f_\lambda(i),&\text{otherwise.}
    \end{cases} \]
  In particular,  $f_{\lambda+\alpha}(j)-f_\rho(j)$ is still the maximal value of $f_{\lambda+\alpha}-f_\rho$, so the proof follows by induction.
\end{proof}

\begin{remark}
  We note that Proposition~\ref{prop:span_A} does not hold for other simply laced root systems. For instance, let $\Phi = D_4$, with the numbering of simple roots as in Figure~\ref{fig:dynkinclassification} (so $\alpha_2$ corresponds to the vertex of the Dynkin diagram of degree $3$). Let us abbreviate the weight $t(\alpha_1+\alpha_3+\alpha_4)+r\alpha_2$ by $\nu_{t,r}$. Consider the weight $\lambda=\nu_{3,6}$. Although $\lambda \in Q$, in fact $\lambda$ does not belong to $\Pi^{\circ,Q}(\rho)$ since in this case $\rho=\nu_{3,5}$. However, the roots $\alpha_1,\alpha_3,\alpha_4$ are all orthogonal to $\lambda$ and to each other so we can fire them to get to $\nu_{4,6}$, which is then orthogonal to $\alpha_2$. Thus all simple roots belong to $\SpanPP(\lambda)$ and so $\lambda$ is connected even though it is outside $\Pi^{\circ,Q}(\rho)$.
\end{remark}

\subsection{Interpretation of connectedness in terms of chips}\label{sec:interpr-terms-chips}

Let us translate the notions from this section to the language of chips. So for the remainder of this section, we assume that $\Phi = A_{N-1}$. We will see that the classification of connected weights in Type A leads to an interesting procedure for computing stabilizations of unlabeled chip configurations on a line.

Consider the map $\sumcoordmap:\mathbb{R}^N\to\mathbb{R}$ defined by $\sumcoord{v}=v_1+v_2+\dots+v_N$. Recall that we may identify $V$ with the $(N-1)$-dimensional space $\{v\in\mathbb{R}^N\colon \sumcoord{v}=0\}$. We denote the standard basis vectors of $\mathbb{R}^N$ by $e_1,\dots,e_N$ and we let $\allones=\frac1N(e_1+e_2+\dots+e_N)$. Each simple root $\alpha_i$ for $1\leq i< N$ equals $e_{i}-e_{i+1}\in V$.
The fundamental weight $\omega_i$ has coordinates $e_1+e_2+\dots+e_i-i \allones$, and thus the weight lattice $P$ is given by
\[P=\{a_0\allones+a_1e_1+\dots+a_Ne_N\colon a_0,\dots,a_N\in\mathbb{Z}: a_0+a_1+\dots+a_N=0\}. \]

As in Section~\ref{sec:labeled-chip-firing-B-C-D}, a (labeled) chip configuration is a vector $v\in \mathbb{Z}^N$: this vector corresponds to the $i$-th chip being at position $v_i$. Each labeled chip configuration~$v$ corresponds to a weight $\lambda(v)=(\lambda_1,\dots,\lambda_N)$ defined by $\lambda(v)=v-\sumcoord v \allones$. An \emph{unlabeled chip configuration} is a configuration $v\in \mathbb{Z}^N$ whose entries are weakly decreasing: $v_1\geq v_2\geq\dots\geq v_N$. Thus we view unlabeled chip configurations as labeled chip configurations whose labeling are weakly decreasing from left to right. Clearly, $v\in\mathbb{Z}^N$ is such a configuration if and only if the weight~$\lambda(v)$ is dominant. For the rest of this section, we assume all chip configurations to be unlabeled.

A chip configuration $v\in\mathbb{Z}^N$ is stable if it does not have two chips located in the same position, equivalently, if $\lambda(v)$ is a strictly dominant weight. We say that a stable configuration $v\in\mathbb{Z}^N$ \emph{has at most one gap} if there is at most one position $t\in\mathbb{Z}$ such that $v_N<t<v_1$ but $t\neq v_i$ for any $1\leq i\leq N$. The weight $\rho$ has coordinates $(N-1,N-2,\dots,0)-\binom{N}{2}\allones$, so any $v\in\mathbb{Z}^N$ with $\lambda(v)=\rho$ is a stable configuration of~$N$ chips with no gaps. All the fundamental weights are minuscule and for $1\leq i<N$, $\rho+\omega_i$ corresponds to a stable chip configuration with exactly one gap.

Given an unlabeled chip configuration $v\in\mathbb{Z}^N$, we define its \emph{pseudo-stabilization} $\pstv$ as follows: $\pstv$ is the unique unlabeled stable configuration with at most one gap such that $\sumcoord{\pstv}=\sumcoord v$, i.e., such that $v$ and $\pstv$ have the same \emph{center of mass}. The motivation for this definition is the following simple observation: for any unlabeled chip configuration $v\in\mathbb{Z}^N$ we have
\begin{equation}\label{eq:pseudostab}
\lambda(\pstv)=\rho+\omega,
\end{equation}
where $\omega \in \Omega^0_m$ is such that $\lambda(v)-\rho \in Q+\omega$.

\def\f{f}
For $v\in \mathbb{Z}^N$, define $\f_v:[N]\to \mathbb{Z}$ by $\f_v(i)=v_1+\dots+v_i$. Note that $\f_v(N)=\sumcoord v$. Given two unlabeled chip configurations $u,v\in\mathbb{Z}^N$ with the same center of mass $\sumcoord u=\sumcoord v$, we write $u\domin v$ if for all $1\leq i\leq N$ we have $\f_u(i)\leq f_v(i)$. We write $u\sdomin v$ if for  all~$1\leq i<N$ we have $f_u(i)< f_v(i)$. 

\begin{prop}
  Given any unlabeled chip configuration $v\in\mathbb{Z}^n$, the dominant weight $\lambda(v)$ is connected if and only if $v\sdomin \pstv$.
\end{prop}
\begin{proof}
  This follows from Proposition~\ref{prop:span_A} together with~\eqref{eq:pseudostab} and the observation that the facet inequalities describing $\Pi^{\circ}(\rho+\omega)$ at $\rho+\omega$ are precisely~$\<\lambda(v),\omega_i\><\<\rho+\omega,\omega_i\>$ for all~$1\leq i<N$.
\end{proof}

The following corollary can be easily deduced from Proposition~\ref{prop:unlabeled_stabilization} combined with the proof of Proposition~\ref{prop:span_A}.

\begin{cor}\label{cor:unlabeled_connected_chip}
  Suppose that $v\in\mathbb{Z}^N$ is an unlabeled chip configuration (and so the chips in $v$ are labeled from left to right in weakly decreasing order). If $v\sdomin \pstv$ then the stabilization of $v$ is $\pstv$. Otherwise, if $1\leq j< N$ is the index that maximizes the quantity
  \[v_1+v_2+\dots+v_j- \left(\pstv_1+\pstv_2+\dots+\pstv_j \right)\]
  then the chips with labels $j$ and $j+1$ can never fire together in any (unlabeled) chip-firing sequence starting from $v$.
\end{cor}

\begin{example}
  Let $N=11$. Consider the configuration $v=(8,8,8,8,4,3,3,0,0,0,0)$ shown in Figure~\ref{fig:chips_unlabeled} (top). We compute
  \[\pstv =(9,8,7,6,5,4,3,2,1,-1,-2)\]
  shown in Figure~\ref{fig:chips_unlabeled} (middle).
  Taking the partial sums gives us the functions $\f_v$ and $\f_\pstv$:
  \begin{equation*}
    \begin{split}
  \f_v   &=(8,16,24,32,36,39,42,42,42,42,42),\\
  \f_\pstv&=(9,17,24,30,35,39,42,44,45,44,42).
    \end{split}
  \end{equation*}
  Here we write $\f_v=(\f_v(1),\f_v(2),\dots,\f_v(N))$. Note that \[\f_v(N)=\f_\pstv(N)=\sumcoord{v}=\sumcoord{\pstv}.\]
  
  Comparing $\f_v$ with $\f_\pstv$, we see that $\lambda(v)$ is not connected and the index $j$ from Corollary~\ref{cor:unlabeled_connected_chip} is equal to $4$ with $\f_v(4)-\f_\pstv(4)=32-30=2$. One can check directly that chips~\chip4 and~\chip5 in Figure~\ref{fig:chips_unlabeled} can never fire together in any unlabeled chip-firing sequence starting from $v$. 

Let us now split $v$ into two configurations  $v'=(8,8,8,8)$ and $v''=(4,3,3,0,0,0,0)$ with $N'=4$ and $N''=7$ chips respectively. One can compute that $\pstv'=(10,9,7,6)$ and~$\pstv''=(5,4,3,1,0,-1,-2)$. We then have
  \begin{equation*}
    \begin{split}
  \f_{v'}&=(8,16,24,32),\\
  \f_{\pstv'}&=(10,19,26,32),
    \end{split}
  \end{equation*}
  and
    \begin{equation*}
    \begin{split}
  \f_{v''}&=(4,7,10,10,10,10,10),\\
  \f_{\pstv''}&=(5,9,12,13,13,12,10).
\end{split}
\end{equation*}

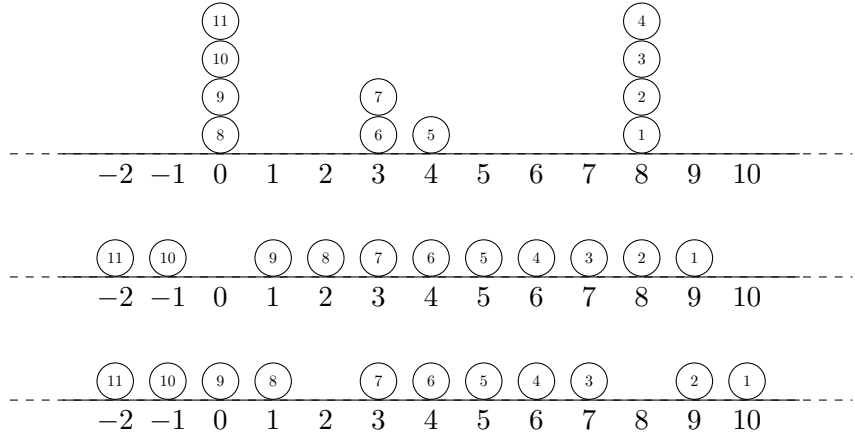
\begin{figure}

  \def\chiptikzscl{0.7}
  \def\chipscl{0.6}
  \def\chipcoef{1.2}
  \def\dashedfrom{-4}
  \def\dashedto{12}
  \def\solidfrom{-3}
  \def\solidto{11}
  \begin{tabular}{c}
  \begin{tikzpicture}[scale=\chiptikzscl,block/.style={draw,circle, minimum width={width("11")+12pt},
font=\small,scale=\chipscl}]
    \foreach \x in {-2,-1,...,10} {%
      \node[anchor=north] (A\x) at (\x,0) {$\x$};
    }
    \draw[dashed] (\dashedfrom,0) -- (\dashedto,0);
    \draw (\solidfrom,0) -- (\solidto,0);
    \foreach[count=\i] \a/\b in {8/1,8/2,8/3,8/4,4/1,3/1,3/2,0/1,0/2,0/3,0/4} {%
      \node[block] at (\a,{\b*\chipscl*\chipcoef-0.5*\chipscl*\chipcoef}) {$\i$};
      }
  \end{tikzpicture}\\ \\
  \begin{tikzpicture}[scale=\chiptikzscl,block/.style={draw,circle, minimum width={width("11")+12pt},
font=\small,scale=\chipscl}]
    \foreach \x in {-2,-1,...,10} {%
      \node[anchor=north] (A\x) at (\x,0) {$\x$};
    }
    \draw[dashed] (\dashedfrom,0) -- (\dashedto,0);
    \draw (\solidfrom,0) -- (\solidto,0);
    \foreach[count=\i] \a/\b in {9/1,8/1,7/1,6/1,5/1,4/1,3/1,2/1,1/1,-1/1,-2/1} {%
      \node[block] at (\a,{\b*\chipscl*\chipcoef-0.5*\chipscl*\chipcoef}) {$\i$};
      }
    \end{tikzpicture}\\ \\
    \begin{tikzpicture}[scale=\chiptikzscl,block/.style={draw,circle, minimum width={width("11")+12pt},
font=\small,scale=\chipscl}]
    \foreach \x in {-2,-1,...,10} {%
      \node[anchor=north] (A\x) at (\x,0) {$\x$};
    }
    \draw[dashed] (\dashedfrom,0) -- (\dashedto,0);
    \draw (\solidfrom,0) -- (\solidto,0);
    \foreach[count=\i] \a/\b in {10/1,9/1,7/1,6/1,5/1,4/1,3/1,1/1,0/1,-1/1,-2/1} {%
      \node[block] at (\a,{\b*\chipscl*\chipcoef-0.5*\chipscl*\chipcoef}) {$\i$};
      }
  \end{tikzpicture}

  \end{tabular}
  \caption{\label{fig:chips_unlabeled} The chip configuration $v$ (top), its pseudo-stabilization $\pstv$ (middle), and its actual stabilization (bottom).}
\end{figure}

Since $\f_{v'}(i)<\f_{\pstv'}(i)$ for all $1\leq i<N'$ and $\f_{v''}(i)<\f_{\pstv''}(i)$ for all $1\leq i<N''$, we conclude that both $\lambda(v')$ and $\lambda(v'')$ are connected and thus $v'$ and $v''$ stabilize to $\pstv'$ and $\pstv''$ respectively. Therefore the stabilization of $v$ is the superposition of $\pstv'$ and $\pstv''$, namely, the configuration $(10,9,7,6,5,4,3,1,0,-1,-2)$ shown in Figure~\ref{fig:chips_unlabeled} (bottom). We encourage the reader to check that this is indeed the result of playing the chip-firing game starting from $v$.
\end{example}

\section{Confluence of central-firing}\label{sec:confl-centr-firing}

In this section, we make some progress towards answering Question~\ref{question:centralconf} in certain cases. 

\subsection{The confluence conjecture}

We now formulate our main conjecture that describes the initial points in the set $\Omega\cup\{0\}$ from which central-firing is confluent. It is based on extensive computations.

\begin{conj}\label{conj:master_central}
  Let $\omega\in\Omega\cup\{0\}$ be a fundamental weight or zero. Then $\raPP$ is confluent from $\omega$ if and only if $\omega \notin Q+\rho$, unless one of the four exceptional cases happens:
  \begin{enumerate}[\normalfont(1)]
  \item\label{item:excA} $\Phi = A_{n}$ in which case $\raPP$ is confluent from $\omega$ if and only if
    \begin{equation}\label{eq:confluent_weights_type_A}
     \begin{cases}
        \omega=0,\omega_1,\omega_n, &\text{if $n$ is odd;}\\
        \omega=\omega_{n/2},\omega_{n/2+1}, &\text{if $n$ is even.}\\
      \end{cases}
    \end{equation}
  \item $\Phi = B_n$ in which case $\raPP$ is confluent from $\omega=\omega_n$ despite the fact that $\omega_n\in Q+\rho$;
  \item $\Phi = D_{4n+2}$ for $n \geq 1$ in which case $\raPP$ is not confluent from  $\omega=0$ even though $0 \notin Q+\rho$;
    \item\label{item:excG} $\Phi = G_2$ in which case $\raPP$ is confluent from both $\omega_1$ and $\omega_2$ even though $P=Q$.
    \end{enumerate}
    Here the simple roots are numbered as in Figure~\ref{fig:dynkinclassification}.
\end{conj}

More explicitly, the elements of $\Omega\cup\{0\}$ from which central-firing is confluent for each root system are listed in Table~\ref{tab:central_firing}. In particular, observe that the weights corresponding to the exceptional cases~\eqref{item:excA} -- \eqref{item:excG}, which are highlighted in red and green in the table, are quite rare, especially outside Type A. We have verified Conjecture~\ref{conj:master_central} for all root systems of rank at most $8$, see Proposition~\ref{prop:comp}.

\begin{remark}\label{rmk:confluence_classical}
According to Conjecture~\ref{conj:master_central}, for each pair $(\Phi,\omega)$ shown in Figure~\ref{fig:initial_confs}, $\raPP$ is confluent from $\omega$. We encourage the reader to check that the result of applying the moves to these configurations as described in Section~\ref{sec:labeled-chip-firing-B-C-D} does not depend on the choices made along the way.
\end{remark}

We will spend this section proving various parts of Conjecture~\ref{conj:master_central}. We start by showing that in simply laced cases, having $\omega\equiv\rho$ in $P/Q$ implies that $\raPP$ is not confluent from~$\omega$.
\begin{table}
  \centering
  \makebox[\textwidth]{
\begin{tabular}{|@{}c@{}|}\hline
    \scalebox{\sclbx}{
    \begin{tikzpicture}[scale=\tpscale]
      \node (A) at (-1,0) {$A_{2n}$};
      \nodeN{0}{0}{0,0}{\rho }{black}
      \nodeN{1}{1}{1,0}{ }{red}
      \nodeD{2}{2}{2,0}{ }{red}
      \nodeN{3}{n-1}{3,0}{ }{red}
      \nodeC{4}{n}{4,0}{ }{black}
      \nodeC{5}{n+1}{5,0}{ }{black}
      \nodeN{6}{n+2}{6,0}{ }{red}
      \nodeD{7}{7}{7,0}{ }{red}
      \nodeN{8}{2n}{8,0}{ }{red}
      \draw (N1)--(N2)--(N3)--(N4)--(N5)--(N6)--(N7)--(N8);
      \draw[white] (9,0)--(9.3,0);
    \end{tikzpicture}} \\
   \scalebox{\sclbx}{
    \begin{tikzpicture}[scale=\tpscale]
      \node (A) at (-1,0) {$A_{2n+1}$};
      \nodeC{0}{0}{0,0}{ }{black}
      \nodeC{1}{1}{1,0}{ }{black}
      \nodeN{2}{2}{2,0}{ }{red}
      \nodeD{3}{3}{3,0}{ }{red}
      \nodeN{4}{n}{4,0}{ }{red}
      \nodeN{5}{n+1}{5,0}{\rho}{black}
      \nodeN{6}{n+2}{6,0}{ }{red}
      \nodeD{7}{7}{7,0}{ }{red}
      \nodeN{8}{2n}{8,0}{ }{red}
      \nodeC{9}{2n+1}{9,0}{ }{black}
      \draw (N1)--(N2)--(N3)--(N4)--(N5)--(N6)--(N7)--(N8)--(N9);
    \end{tikzpicture}} \\\hline
  \scalebox{\sclbx}{
    \begin{tikzpicture}[scale=\tpscale,decoration={markings,mark=at position 0.7 with {\arrow{>}}}]
      \node (A) at (-1,0) {$B_{n}$};
      \nodeC{0}{0}{0,0}{ }{black}
      \nodeC{1}{1}{1,0}{ }{black}
      \nodeC{2}{2}{2,0}{ }{black}
      \nodeD{3}{3}{3,0}{ }{black}
      \nodeC{4}{n-1}{4,0}{ }{black}
      \nodeC{5}{n}{5,0}{\rho}{green!\shadeofgrey!black}
      \draw (N1)--(N2)--(N3)--(N4);
      \draw[double,postaction={decorate}] (N4) -- (N5);
      \nodeC{4}{n-1}{4,0}{ }{black}
      \nodeC{5}{n}{5,0}{\rho}{green!\shadeofgrey!black}
      \draw[white] (6,0)--(9.2,0);
    \end{tikzpicture}}\\\hline
  \scalebox{\sclbx}{
    \begin{tikzpicture}[scale=\tpscale,decoration={markings,mark=at position 0.7 with {\arrow{>}}}]
      \node (A) at (-1,0) {$C_{4n}$};
      \nodeN{0}{0}{0,0}{\rho}{black}
      \nodeC{1}{1}{1,0}{ }{black}
      \nodeN{2}{2}{2,0}{\rho}{black}
      \nodeC{3}{3}{3,0}{ }{black}
      \nodeN{4}{4}{4,0}{\rho}{black}
      \nodeD{5}{5}{5,0}{ }{black}
      \nodeC{6}{4n-3}{6,0}{ }{black}
      \nodeN{7}{4n-2}{7,0}{\rho}{black}
      \nodeC{8}{4n-1}{8,0}{ }{black}
      \nodeN{9}{4n}{9,0}{\rho}{black}
      \draw (N1)--(N2)--(N3)--(N4)--(N5)--(N6)--(N7)--(N8);
      \draw[double,postaction={decorate}] (N9) -- (N8);
      \nodeC{8}{4n-1}{8,0}{ }{black}
      \nodeN{9}{4n}{9,0}{\rho}{black}
    \end{tikzpicture}}\\
  \scalebox{\sclbx}{
    \begin{tikzpicture}[scale=\tpscale,decoration={markings,mark=at position 0.7 with {\arrow{>}}}]
      \node (A) at (-1,0) {$C_{4n+1}$};
      \nodeC{0}{0}{0,0}{ }{black}
      \nodeN{1}{1}{1,0}{\rho}{black}
      \nodeC{2}{2}{2,0}{ }{black}
      \nodeN{3}{3}{3,0}{\rho}{black}
      \nodeC{4}{4}{4,0}{ }{black}
      \nodeD{5}{5}{5,0}{ }{black}
      \nodeC{6}{4n-2}{6,0}{ }{black}
      \nodeN{7}{4n-1}{7,0}{\rho}{black}
      \nodeC{8}{4n}{8,0}{ }{black}
      \nodeN{9}{4n+1}{9,0}{\rho}{black}
      \draw (N1)--(N2)--(N3)--(N4)--(N5)--(N6)--(N7)--(N8);
      \draw[double,postaction={decorate}] (N9) -- (N8);
      \nodeC{8}{4n}{8,0}{ }{black}
      \nodeN{9}{4n+1}{9,0}{\rho}{black}
    \end{tikzpicture}}\\
  \scalebox{\sclbx}{
    \begin{tikzpicture}[scale=\tpscale,decoration={markings,mark=at position 0.7 with {\arrow{>}}}]
      \node (A) at (-1,0) {$C_{4n+2}$};
      \nodeC{0}{0}{0,0}{ }{black}
      \nodeN{1}{1}{1,0}{\rho}{black}
      \nodeC{2}{2}{2,0}{ }{black}
      \nodeN{3}{3}{3,0}{\rho}{black}
      \nodeC{4}{4}{4,0}{ }{black}
      \nodeD{5}{5}{5,0}{ }{black}
      \nodeN{6}{4n-1}{6,0}{\rho}{black}
      \nodeC{7}{4n}{7,0}{ }{black}
      \nodeN{8}{4n+1}{8,0}{\rho}{black}
      \nodeC{9}{4n+2}{9,0}{ }{black}
      \draw (N1)--(N2)--(N3)--(N4)--(N5)--(N6)--(N7)--(N8);
      \draw[double,postaction={decorate}] (N9) -- (N8);
      \nodeN{8}{4n+1}{8,0}{\rho}{black}
      \nodeC{9}{4n+2}{9,0}{ }{black}
    \end{tikzpicture}}\\
    \scalebox{\sclbx}{
    \begin{tikzpicture}[scale=\tpscale,decoration={markings,mark=at position 0.7 with {\arrow{>}}}]
      \node (A) at (-1,0) {$C_{4n+3}$};
      \nodeN{0}{0}{0,0}{\rho}{black}
      \nodeC{1}{1}{1,0}{ }{black}
      \nodeN{2}{2}{2,0}{\rho}{black}
      \nodeC{3}{3}{3,0}{ }{black}
      \nodeN{4}{4}{4,0}{\rho}{black}
      \nodeD{5}{5}{5,0}{ }{black}
      \nodeN{6}{4n}{6,0}{\rho}{black}
      \nodeC{7}{4n+1}{7,0}{ }{black}
      \nodeN{8}{4n+2}{8,0}{\rho}{black}
      \nodeC{9}{4n+3}{9,0}{ }{black}
      \draw (N1)--(N2)--(N3)--(N4)--(N5)--(N6)--(N7)--(N8);
      \draw[double,postaction={decorate}] (N9) -- (N8);
      \nodeN{8}{4n+2}{8,0}{\rho}{black}
      \nodeC{9}{4n+3}{9,0}{ }{black}
    \end{tikzpicture}}\\\hline
  \scalebox{\sclbx}{
    \begin{tikzpicture}[scale=\tpscale,decoration={markings,mark=at position 0.7 with {\arrow{>}}}]
      \node (A) at (-1,0) {$D_{4n}$};
      \nodeN{0}{0}{0,0}{\rho}{black}
      \nodeC{1}{1}{1,0}{ }{black}
      \nodeN{2}{2}{2,0}{\rho}{black}
      \nodeC{3}{3}{3,0}{ }{black}
      \nodeN{4}{4}{4,0}{\rho}{black}
      \nodeD{5}{5}{5,0}{ }{black}
      \nodeC{6}{4n-3}{6,0}{ }{black}
      \nodeN{7}{4n-2}{7,0}{\rho }{black}
      \nodeC{8}{4n-1}{8,\dheight}{ }{black}
      \nodeC{9}{4n}{8,-\dheight}{ }{black}
      \draw (N1)--(N2)--(N3)--(N4)--(N5)--(N6)--(N7)--(N8);
      \draw (N7)--(N9);
    \end{tikzpicture}}\\
\scalebox{\sclbx}{
    \begin{tikzpicture}[scale=\tpscale,decoration={markings,mark=at position 0.7 with {\arrow{>}}}]
      \node (A) at (-1,0) {$D_{4n+1}$};
      \nodeN{0}{0}{0,0}{\rho}{black}
      \nodeC{1}{1}{1,0}{ }{black}
      \nodeN{2}{2}{2,0}{\rho}{black}
      \nodeC{3}{3}{3,0}{ }{black}
      \nodeN{4}{4}{4,0}{\rho}{black}
      \nodeD{5}{5}{5,0}{ }{black}
      \nodeN{6}{4n-2}{6,0}{\rho}{black}
      \nodeC{7}{4n-1}{7,0}{ }{black}
      \nodeC{8}{4n}{8,\dheight}{ }{black}
      \nodeC{9}{4n+1}{8,-\dheight}{ }{black}
      \draw (N1)--(N2)--(N3)--(N4)--(N5)--(N6)--(N7)--(N8);
      \draw (N7)--(N9);
    \end{tikzpicture}}\\
\scalebox{\sclbx}{
    \begin{tikzpicture}[scale=\tpscale,decoration={markings,mark=at position 0.7 with {\arrow{>}}}]
      \node (A) at (-1,0) {$D_{4n+2}$};
      \nodeN{0}{0}{0,0}{ }{red}
      \nodeN{1}{1}{1,0}{\rho}{black}
      \nodeC{2}{2}{2,0}{ }{black}
      \nodeN{3}{3}{3,0}{\rho}{black}
      \nodeC{4}{4}{4,0}{ }{black}
      \nodeD{5}{5}{5,0}{ }{black}
      \nodeN{6}{4n-1}{6,0}{\rho}{black}
      \nodeC{7}{4n}{7,0}{ }{black}
      \nodeC{8}{4n+1}{8,\dheight}{ }{black}
      \nodeC{9}{4n+2}{8,-\dheight}{ }{black}
      \draw (N1)--(N2)--(N3)--(N4)--(N5)--(N6)--(N7)--(N8);
      \draw (N7)--(N9);
    \end{tikzpicture}}\\
\scalebox{\sclbx}{
    \begin{tikzpicture}[scale=\tpscale,decoration={markings,mark=at position 0.7 with {\arrow{>}}}]
      \node (A) at (-1,0) {$D_{4n+3}$};
      \nodeC{0}{0}{0,0}{ }{black}
      \nodeN{1}{1}{1,0}{\rho}{black}
      \nodeC{2}{2}{2,0}{ }{black}
      \nodeN{3}{3}{3,0}{\rho}{black}
      \nodeC{4}{4}{4,0}{ }{black}
      \nodeD{5}{5}{5,0}{ }{black}
      \nodeC{6}{4n}{6,0}{ }{black}
      \nodeN{7}{4n+1}{7,0}{\rho}{black}
      \nodeC{8}{4n+2}{8,\dheight}{ }{black}
      \nodeC{9}{4n+3}{8,-\dheight}{ }{black}
      \draw (N1)--(N2)--(N3)--(N4)--(N5)--(N6)--(N7)--(N8);
      \draw (N7)--(N9);
    \end{tikzpicture}}\\\hline
\begin{tabular}{c|c}
\scalebox{\sclbx}{
    \begin{tikzpicture}[scale=\tpscale,decoration={markings,mark=at position 0.7 with {\arrow{>}}}]
      \node (A) at (-1,0) {$E_{6}$};
      \nodeN{0}{0}{0,0}{\rho}{black}
      \nodeC{1}{1}{1,0}{ }{black}
      \nodeC{3}{3}{2,0}{ }{black}
      \nodeC{5}{5}{4,0}{ }{black}
      \nodeC{6}{6}{5,0}{ }{black}
      
      \nodeN{4}{4}{3,0}{ }{black}
\def\sw{west}
\def\ne{east}
      \nodeN{2}{2}{3,0.5}{\rho}{black}
\def\sw{south west}
\def\ne{north east}
      \nodeN{4}{ }{3,0}{\rho}{black}
\def\sw{south}
\def\ne{north}
      \draw (N1)--(N3)--(N4)--(N5)--(N6);
      \draw (N4)--(N2);
      \draw[white] (6,0)--(6.2,0);
    \end{tikzpicture}} &

\scalebox{\sclbx}{
\begin{tikzpicture}[scale=\tpscale,decoration={markings,mark=at position 0.7 with {\arrow{>}}}]
      \node (A) at (-1,0) {$F_4$};
      \nodeN{0}{0}{0,0}{\rho }{black}
      \nodeN{1}{1}{1,0}{\rho }{black}
      \nodeN{2}{2}{2,0}{\rho }{black}
      \nodeN{3}{3}{3,0}{\rho }{black}
      \nodeN{4}{4}{4,0}{\rho }{black}
      \draw (N1)--(N2);
      \draw (N3)--(N4);
      \draw[double,postaction={decorate}] (N2) -- (N3);
      \nodeN{2}{2}{2,0}{\rho }{black}
      \nodeN{3}{3}{3,0}{\rho }{black}
    \end{tikzpicture}} \\\cline{2-2}
  \scalebox{\sclbx}{
    \begin{tikzpicture}[scale=\tpscale,decoration={markings,mark=at position 0.7 with {\arrow{>}}}]
      \node (A) at (-1,0) {$E_{7}$};
      \nodeC{0}{0}{0,0}{ }{black}
      \nodeC{1}{1}{1,0}{ }{black}
      \nodeC{3}{3}{2,0}{ }{black}
      \nodeN{5}{5}{4,0}{\rho}{black}
      \nodeC{6}{6}{5,0}{ }{black}
      \nodeN{7}{7}{6,0}{\rho}{black}
      
      \nodeC{4}{4}{3,0}{ }{black}
\def\sw{west}
\def\ne{east}
      \nodeN{2}{2}{3,0.5}{\rho}{black}
\def\sw{south}
\def\ne{north}
      \draw (N1)--(N3)--(N4)--(N5)--(N6)--(N7);
      \draw (N4)--(N2);
    \end{tikzpicture}} &
\scalebox{\sclbx}{
\begin{tikzpicture}[scale=\tpscale,decoration={markings,mark=at position 0.7 with {\arrow{>}}}]
      \node (A) at (-1,0) {$G_2$};
      \nodeN{0}{0}{0,0}{\rho }{black}
      \nodeC{1}{1}{1,0}{\rho }{green!\shadeofgrey!black}
      \nodeC{2}{2}{2,0}{\rho }{green!\shadeofgrey!black}
      \draw[double distance=1.5pt,postaction={decorate}] (N2) -- (N1);
      \draw (N1) -- (N2);
      \nodeC{1}{1}{1,0}{\rho }{green!\shadeofgrey!black}
      \nodeC{2}{2}{2,0}{\rho }{green!\shadeofgrey!black}
    \end{tikzpicture}} \\\cline{2-2}
\end{tabular}\\
\scalebox{\sclbx}{
    \begin{tikzpicture}[scale=\tpscale,decoration={markings,mark=at position 0.7 with {\arrow{>}}}]
      \node (A) at (-1,0) {$E_{8}$};
      \nodeN{0}{0}{0,0}{\rho}{black}
      \nodeN{1}{1}{1,0}{\rho}{black}
      \nodeN{3}{3}{2,0}{\rho}{black}
      \nodeN{5}{5}{4,0}{\rho}{black}
      \nodeN{6}{6}{5,0}{\rho}{black}
      \nodeN{7}{7}{6,0}{\rho}{black}
      \nodeN{8}{8}{7,0}{\rho}{black}
      
      \nodeN{4}{4}{3,0}{ }{black}
\def\sw{west}
\def\ne{east}
      \nodeN{2}{2}{3,0.5}{\rho}{black}
\def\sw{south west}
\def\ne{north east}
      \nodeN{4}{ }{3,0}{\rho}{black}
\def\sw{south}
\def\ne{north}
      \draw (N1)--(N3)--(N4)--(N5)--(N6)--(N7)--(N8);
      \draw (N4)--(N2);
      \draw[white] (8,0) -- (12.2,0);
    \end{tikzpicture}}\\\hline
 \end{tabular}}
\caption{Confluence of central-firing from weights in $\Omega\cup\{0\}$. A vertex corresponding to $0$, resp., $\omega_i$ is labeled by $0$, resp., $i$ (as in Figure~\ref{fig:dynkinclassification}). Weights from which central-firing is confluent correspond to filled vertices with boldface labels. If~$\omega \in Q+\rho$ then the corresponding vertex is marked by $\rho$. If~$\omega\notin Q+\rho$ but central-firing is still not confluent from~$\omega$ then it is colored red. If~$\omega \in Q+\rho$ but central-firing is confluent from~$\omega$ then it is colored green.}\label{tab:central_firing}
\end{table}

\begin{prop}
Suppose that $\Phi$ is simply laced. Let $\lambda \in P$ be a dominant weight that belongs to $\Pi^Q(\rho)$ but is not equal to $\rho$. Then~$\raPP$ is not confluent from $\lambda$.
\end{prop}
\begin{proof}
  We know from Propositions~\ref{prop:unlabeled_stabilization} and~\ref{prop:simply_laced_dominant}  that there exists a firing sequence 
  \[\lambda=\lambda_0\raPP\lambda_1\raPP\dots\raPP\lambda_t\raPP\lambda_{t+1}:=\rho\]
  such that for each $0\leq s\leq t+1$, $\lambda_s$ is a dominant weight. Let $\alpha\in\Phi^+$ be such that $\lambda_t+\alpha=\rho$. In particular, we have $\<\lambda_t,\alpha^\vee\>=0$ and thus $\<\rho,\alpha^\vee\>=2$. Write $\alpha$ in the basis of simple roots:
  \[\alpha=\sum_{i=1}^{n}a_i\alpha_i.\]
  Since $\rho$ is the sum of the fundamental weights, we get $\sum_{i=1}^{n}a_i=2$ (note that this conclusion uses the fact that $\Phi$ is simply laced). Since $2\alpha_i\notin\Phi$, we get that $\alpha=\alpha_i+\alpha_j$ for some $i\neq j\in [n]$. Moreover, it must be the case that $i$ and $j$ are connected by an edge in the Dynkin diagram $\DD$ of $\Phi$ because otherwise $\alpha_i+\alpha_j$ would not be a root. Thus $\<\alpha_i,\alpha_j^\vee\>= -1$. Let us now consider the weight $\rho=\lambda_t+\alpha_i$. We claim that $\rho$ is $\raPP$-stable and that $\<\lambda_t,\alpha_i^\vee\>=0$, that is, $\lambda_t\raPP \rho$. Indeed, we have
  \[\<\lambda_t,\alpha_i^\vee\>=\<\rho-\alpha_i-\alpha_j,\alpha_i^\vee\>=1-2+1=0.\]
  Thus  $\lambda_t\raPP \rho$. On the other hand, $\rho=\rho-\alpha_j$ is a vertex of $\Pi^Q(\rho)$:
  \[s_{\alpha_j}(\rho)=\rho-\<\rho,\alpha_j^\vee\>\alpha_j=\rho.\]
  In particular, it is $\raPP$-stable.
\end{proof}

This proposition immediately implies some parts of Conjecture~\ref{conj:master_central}:

\begin{cor}\label{cor:simply_laced_not_confluent}
Suppose $\Phi$ is simply laced and of rank greater than one. Let $\omega\in\Omega\cup\{0\}$ be such that $\omega \in Q+\rho$. Then central-firing is not confluent from~$\omega$.
\end{cor}

Note that in Types $B_2$ and $G_2$, the result of this corollary is false (and assuming Conjecture~\ref{conj:master_central}, it is false for $B_n$ for all $n\geq 2$), so the simply laced requirement is necessary. On the other hand, the result of this corollary still appears to hold for $C_n$ and holds for $F_4$.

For root systems of small rank, we have verified the conjecture using a computer (in fact, the computation finishes in a reasonable amount of time).

\begin{prop}\label{prop:comp}
Conjecture~\ref{conj:master_central} holds for all root systems $\Phi$ of rank at most $8$.
\end{prop}
This includes all root systems of exceptional types.

Let us also mention some results of Hopkins-McConville-Propp~\cite{hopkins2017sorting}: 

\begin{thm}[\cite{hopkins2017sorting}] \label{thm:centralconfab}
  Conjecture~\ref{conj:master_central} is true for $\omega=0$ when $\Phi=A_n$ or~$B_n$.
\end{thm}

Actually, when $\Phi$ is of Type $B_n$, it is easy to see that for any $\omega \in \Omega\setminus \{\omega_n\}$, we have that~$0 \raPPAst \omega$. So Theorem~\ref{thm:centralconfab} implies almost all cases of Conjecture~\ref{conj:master_central} for Type~B:

\begin{cor}
  Conjecture~\ref{conj:master_central} is true for $\omega\in\Omega\setminus\{\omega_n\}$ when $\Phi = B_n$.
\end{cor}

As we have already mentioned in Remark~\ref{rmk:half_integers}, the confluence of central-firing from $\omega_n$ in Type $B_n$ is equivalent to the confluence of central-firing from $\omega_n$ in Type $D_n$. This case remains open.

When $\Phi=B_n$, we offer the following extension of Conjecture~\ref{conj:master_central} to a much wider class of weights:

\begin{conj} \label{conj:typeb}
Suppose that $\Phi=B_n$. Let $\lambda\in P$ be a connected, dominant weight. Then central-firing is confluent from~$\lambda$.
\end{conj}

Note that the connectedness assumption in Conjecture~\ref{conj:typeb} is clearly required: otherwise, one can just choose $\lambda$ to be far enough from the origin so that the only roots in the firing span of $\lambda$ form a sub-root system of Type $A_2$. One can construct an example showing the dominance requirement is also necessary already for $n=3$. Note also that Conjecture~\ref{conj:typeb} has a counterpart in Type D for connected, dominant weights $\l\in P$ such that $\l\equiv\omega_n$ or $\l\equiv\omega_{n-1}$ modulo $Q$ (see Remark~\ref{rmk:half_integers}).

Finally, let us show that all the red vertices in the Type A part of Table~\ref{tab:central_firing} really are non-confluent.

\begin{prop}
Suppose that $\Phi =A_{N-1}$ and consider a weight $\omega\in\init$. Then~$\raPP$ is not confluent from $\omega$ unless $\omega$ is given by~\eqref{eq:confluent_weights_type_A} (in which case it may or may not be confluent).
\end{prop}
\begin{proof}
  The case $\omega\equiv\rho$ modulo $Q$ follows from Corollary~\ref{cor:simply_laced_not_confluent}, thus we may assume that~$\omega\neq 0$, and let $1\leq i< N$ be such that $\omega=\omega_i$. 

  Let us use the chip interpretation of $\omega_i$ from Section~\ref{sec:labeled-chip-firing-B-C-D}. We get that chip~\chip1 is at position~$1$ while chip~\chip N is at the origin. Denote by $v$ and $v'\in\R^{N-1}$ the chip configurations obtained from $\omega$ by removing chips~\chip1 and~\chip{N}, respectively. Thus we have $v=\omega_{i-1}$, $v'=\omega_i$ for $\Phi'=A_{N-2}$ (except that when $i=N-1$, we have $v'=0$).  Using~\eqref{eq:pseudostab}, we can find the pseudo-stabilizations $\pseudostab{v}$ and $\pseudostab{v}'$  of $v$ and $v'$, respectively. If~$\omega_i$ is not given by~\eqref{eq:confluent_weights_type_A} then it is straightforward to check that the pseudo-stabilization of $v$ (resp., of~$v'$) will necessarily have a gap at some position~$j\in\Z$ (resp., $j'\in\Z$). Moreover, since the pseudo-stabilization of $v$ (resp., $v'$) is required to have the same center of mass as~$v$ (resp., as~$v'$), we see that $j'=j-1$. By Corollary~\ref{cor:unlabeled_connected_chip}, the (unlabeled) stabilization of $v$ (resp., $v'$) coincides with its pseudo-stabilization $\pseudostab{v}$ (resp.,~$\pseudostab{v}'$).

  We now add chip~\chip1 (resp., chip~\chip{N}) back to $\pseudostab{v}$ (resp., $\pseudostab{v}'$) and denote by $w$ (resp., $w'$) any labeled stabilization of the corresponding chip configuration. Suppose that $j>0$. We claim that the chip configurations $w$ and $\pstv$ coincide in positions $j+1,j+2,\dots$. This can be seen either by directly doing the rest of the chip-firing moves, or by applying Corollary~\ref{cor:unlabeled_connected_chip}. Thus chip~\chip1 ends up in position~$j$. Suppose now that $j\leq 0$. Then an analogous argument shows that the chip configurations $w'$ and $\pseudostab{v}'$ coincide in positions $j-2,j-3,\dots$, and therefore chip~\chip{N} ends up in position $j-1$.  In either case, the final configuration ($w$ or $w'$) will not correspond to a dominant weight. However, we know by Proposition~\ref{prop:simply_laced_dominant} that there is also a firing sequence that starts from $\omega_i$ and always stays inside the dominant chamber. Thus we have found two different stabilizations of~$\omega_i$.
\end{proof}

\def\jbar{{\overline j}}
\def\alphaprime{\alpha}
\def\omegaprime{\omega}
\def\Phiprime{\Phi'}
\subsection{Folding}

In this subsection, we quickly explain how one can deduce confluence in a non simply laced system via the \emph{folding} technique, as described for instance in~\cite{stembridge2008folding}. Suppose we are given a simply laced root system $\Phi\subseteq V$ with Dynkin diagram $\DD$ and an automorphism $\sigma:[n]\to [n]$ of $\DD$ that does not send a vertex to its neighbor. From this data, one constructs another root system $\Phiprime$ as follows. Let $J$ be the set of equivalence classes of $[n]$ modulo $\sigma$. For each $\jbar\in J$, define the $\jbar$-th simple root $\alphaprime_\jbar$ of $\Phiprime$ to be the sum of the corresponding simple roots of $\Phi$ (which are necessarily orthogonal to each other):
\[\alphaprime_{\jbar}:=\sum_{i\in\jbar}\alpha_i.\]
It turns out that $\{\alphaprime_\jbar\colon \jbar\in J\}$ is a set of simple roots of another root system $\Phiprime$ whose Dynkin diagram is obtained from $\DD$ via \emph{folding along $\sigma$}. Note that $\Phiprime$ is naturally living inside $V^\sigma:=\{v\in V\colon \sigma(v)=v\}$. Here we extended $\sigma$ to a map $V\to V$ by linearity from its action on simple roots. The fundamental weights $\omegaprime_\jbar$ for $\Phiprime$ are again given by a similar expression:
\[\omegaprime_\jbar:=\sum_{i\in\jbar}\omega_i.\]
It is easy to check that indeed $\<\omegaprime_{\jbar_1},(\alphaprime_{\jbar_2})^\vee\>=\delta_{\jbar_1,\jbar_2}$, where $\delta$ is the Kronecker delta. Thus the weight lattice for $\Phi'$ is $P^\sigma:=\{\l\in P\colon \sigma(\l)=\l\}$.

Let us now discuss the relationship between $\raPP$ and $\raPPprime$. By~\cite[Claim~4]{stembridge2008folding}, each $\sigma$-orbit of $\Phi$ consists of pairwise orthogonal roots. By~\cite[Claim~1]{stembridge2008folding}, the roots of~$\Phiprime$ are precisely of the form $\beta=\sum_{\alpha\in B} \alpha$, where $B$ is a single $\sigma$-orbit of~$\Phi$. Thus if~$\l\raPPprime \l+\beta$ for some $\beta\in\Phiprime^+$ then $\l\raPPAst \l+\beta$ because we can just fire each root in~$B$ in an arbitrary order. We obtain the following result.

\begin{prop}\label{prop:folding}
Suppose that $\raPP$ is confluent from some weight $\l\in P^\sigma$. Then~$\raPPprime$ is confluent from $\l$ as well. 
\end{prop}
\begin{proof}
  Let $\m$ be the unique $\raPP$-stable weight such that $\l\raPPAst \m$. Then $\sigma(\m)$ would also be $\raPP$-stable, and thus we must have $\sigma(\m)=\m$. Suppose that there is some  $\raPPprime$-stable weight $\m'\in P^\sigma$ such that $\l\raPPprimeast \m'$, and assume that $\m'\neq \m$. Then by the above discussion we have that~$\l\raPPAst\m'$ and thus $\m'$ must not be $\raPP$-stable. Thus there is a root $\alpha\in\Phi^+$ such that $\<\m',\alpha^\vee\>=0$. Let $B$ be the $\sigma$-orbit of $\alpha$, then $\beta:=\sum_{\alpha'\in B} \alpha'$ is a positive root for~$\Phiprime$ and since $\m'$ is $\sigma$-invariant, we still have $\<\m',\beta^\vee\>=0$. We have shown that if~$\m'\in P^\sigma$ is an $\raPPprime$-stable weight such that $\l\raPPprimeast\m'$ then $\m'=\m$. Since $\raPPprime$ is terminating, there has to be at least one such stable weight, and thus it follows that $\m$ is the only $\raPPprime$-stable weight that satisfies $\l\raPPprimeast\m$.
\end{proof}

Proposition~\ref{prop:folding} can be directly applied to get some dependencies between various claims in Conjecture~\ref{conj:master_central}. Let us list the most interesting ones:
\begin{itemize}
\item If $\raPP$ is confluent from the origin for $\Phi=A_{2n-1}$ (which it is by the result of~\cite{hopkins2017sorting}) then $\raPP$ is confluent from the origin for $\Phi=B_n$ as well (and indeed a version of folding was essentially what was used in~\cite{hopkins2017sorting} to deduce Type B confluence from Type A);
\item If $\raPP$ is confluent from $\omega_n+\omega_{n+1}$  for $\Phi=D_{n+1}$ then $\raPP$ is confluent from $\omega_n$ for $\Phi=C_n$; 
\item If $\raPP$ is confluent from $0$ (resp., from $\omega_i$ for some $1\leq i<n$) for $\Phi=D_{n+1}$ then $\raPP$ is confluent from $0$ (resp., from $\omega_i$) for $\Phi=C_n$.
\end{itemize}
\begin{remark}
By Proposition~\ref{prop:comp}, central-firing is not confluent from $0$ for $D_6$ even though for~$C_{5}$ it is. (This is generalized in Conjecture~\ref{conj:master_central} to $D_{4n+2}$ and $C_{4n+1}$; but in fact we could not check computationally whether central-firing is confluent from $0$ for~$D_{10}$.) Similarly, one easily checks that central-firing is not confluent from $\omega_2$ for $A_3$ but for $B_2$ it is. Thus the converse to Proposition~\ref{prop:folding} fails to hold in many cases.
\end{remark}

Actually, we can also apply folding to study the connectedness of weights.

\begin{prop}\label{prop:connectedness_folding}
Suppose that $\l\in P^\sigma$ is connected with respect to $\Phi'$. Then~$\l$ is connected with respect to $\Phi$ as well.
\end{prop}
\begin{proof}
By Proposition~\ref{prop:span_w_inv}, we may assume that $\l$ is dominant. Since $\l$ is connected with respect to $\Phi'$, it must be that $\alphaprime_{\jbar} \in \FS_{(\Phi')^+}(\lambda)$ for all simple roots $\alphaprime_{\jbar}$ of $\Phi'$. But since every $\alpha_i$ appears with nonzero coefficient in some $\alphaprime_{\jbar}$, by Proposition~\ref{prop:span_simp} this means that $\alpha_i \in \SpanPP(\lambda)$ for all $i \in [n]$, thus proving the proposition.
\end{proof}

Proposition~\ref{prop:connectedness_folding} for instance lets us apply one direction of Proposition~\ref{prop:span_A} (our classification of connected weights in Type~A) to Type~B as well:

\begin{cor}
Suppose $\Phi=B_n$. Then if the weight $\lambda \in P$ is connected, we have that~$\lambda \in \Pi^{\circ,Q}(\rho+\omega)$ for some $\omega\in \Omega^0_m$. 
\end{cor}
\begin{proof}
Let $\Phi=B_n$ and $\lambda$ be a connected weight of~$\Phi$. By Proposition~\ref{prop:span_w_inv} we may assume that $\lambda$ is dominant. Let $\omega \in \Omega_m^0=\{0,\omega_n\}$ be such that $\rho+\omega-\lambda \in Q$.

Now view $\Phi$ as obtained from $\overline{\Phi}=A_{2n-1}$ via folding as described above. Let $\overline{\rho}$ be the sum of fundamental weights of $\overline{\Phi}$, and $\overline{\omega}$ be the zero-or-minuscule weight of $\overline{\Phi}$ such that $\overline{\rho}+\overline{\omega}-\lambda$ belongs to the root lattice of $\overline{\Phi}$. Note that since the fundamental weights of $\Phi$ are sums of fundamental weights of $\overline{\Phi}$, $\lambda$ is still dominant when considered as a weight of $\overline{\Phi}$. Also, by Proposition~\ref{prop:connectedness_folding}, $\lambda$ is connected when viewed as a weight of~$\overline{\Phi}$. Hence by Proposition~\ref{prop:span_A} together with Lemma~\ref{lemma:permcontainment}, we have that $\overline{\rho}+\overline{\omega}-\lambda$ is a linear combination of simple roots of $\overline{\Phi}$ with strictly positive coefficients. But observe that $\overline{\rho}=\rho$ and also that $\overline{\omega}=\omega$. Therefore, $\rho+\omega-\lambda$ is a linear combination of simple roots of $\overline{\Phi}$ (and hence of $\Phi$) with strictly positive coefficients. The conclusion of the corollary follows by Lemma~\ref{lemma:permcontainment}.
\end{proof}

\subsection{Summary}
We have proved some parts of Conjecture~\ref{conj:master_central}. Let us list all the cases that remain open for $\Phi$ of rank at least $9$ (cf. Proposition~\ref{prop:comp}).
\begin{problem}\leavevmode
  \begin{enumerate}[\normalfont(1)]
  \item Show that $\raPP$ is confluent from $\omega_1$ and $\omega_{2n+1}$ for $\Phi=A_{2n+1}$.
  \item Show that $\raPP$ is confluent from $\omega_n$ and $\omega_{n+1}$ for $\Phi=A_{2n}$.
  \item Show that  $\raPP$ is confluent from $\omega_n$ for $\Phi=B_n$ (equivalently, for $\Phi=D_n$). \label{cond:omegan_b_d}
  \item Show that $\raPP$ is confluent from $\omega\in\init$ if and only if $\omega\not\equiv\rho$ in $P/Q$ for~$\Phi=C_n$.
  \item\label{item:D_4n+2} Show that $\raPP$ is not confluent from $0$ for $\Phi=D_{4n+2}$.
  \item Show that $\raPP$ is confluent from $\omega\in\init$ for all $\omega\not\equiv\rho$ in $P/Q$ for $\Phi=D_n$, except for the case~\eqref{item:D_4n+2} above. \label{cond:omegas_d}
  \end{enumerate}
Note that case~\eqref{cond:omegas_d} includes~\eqref{cond:omegan_b_d} as a special case.
\end{problem}

\bibliography{central_firing}{}
\bibliographystyle{alpha} 

\end{document}